\documentclass[11pt]{article}
\usepackage[utf8]{inputenc}
\usepackage{setspace,mathtools,amsmath,amssymb,amsthm,fullpage,outline,centernot,slashed}
\usepackage{hyperref}
\usepackage{graphicx}
\usepackage{color}
\usepackage{amsmath, amsthm, amssymb, bm}
\usepackage{epstopdf}
\usepackage{abstract}
\usepackage[affil-it]{authblk}

\numberwithin{equation}{section}

\newcommand{\Red}{\textcolor{red}}

\newcommand{\bea}{\begin{eqnarray}}
\newcommand{\eea}{\end{eqnarray}}
\def\beaa{\begin{eqnarray*}}
\def\eeaa{\end{eqnarray*}}
\def\ba{\begin{array}}
\def\ea{\end{array}}
\def\be#1{\begin{equation} \label{#1}}
\def \eeq{\end{equation}}
\def\bsplit{\begin{split}}
\def\lab{\label}
\def\lapp{\slashed{\lap}}
\def\nabb{\slashed\nab}

\def\a{{\alpha}}

\def\b{{\beta}}

\def\be{{\beta}}

\def\Ga{\Gamma}
\def\de{\delta}

\def\ep{\epsilon}

\def\la{\lambda}

\def\si{\sigma}
\def\Si{\Sigma}
\def\om{\omega}

\def\nab{\nabla}

\def\pr{{\partial}}
\def\les{\lesssim}
\def\c{\cdot}

\def\lap{{\triangle}}

\def\R{\mathbb{R}}

\def\Z{\mathbb{Z}}

\def\m{{\bf m}}

\def\p{{\partial}}

\def\SSS{{\mathbb S}}
\def\RRR{{\mathbb R}}

\def\be{\begin{equation}}
\def\ee{\end{equation}}
\def\vo{\vspace{1\baselineskip}}

\def\Div{\mbox{Div}}

\def\Sext{\,^{(ext)}\Si}

\newtheorem{theorem}{Theorem}[section]

\newtheorem{lemma}{Lemma}[section]

\theoremstyle{definition}

\theoremstyle{remark}
\newtheorem{remark}{Remark}[section]

 \numberwithin{equation}{section}

\title{Decay estimates of wave equations in un-isotropic media}
\author[$\dag$]{Sergiu Klainerman  }
\author[$\star$]{Xuecheng Wang }
\affil[$\dag$]{\small Princeton University}
\affil[$\star$]{\small Tsinghua University \& BIMSA}
 
 \date{}

\begin{document}

 \maketitle 

\begin{abstract}
We prove decay estimates  for solutions to   non-isotropic  linear  systems of  wave equations.
The defining feature of these estimates  is that they depend only on  the commutation properties of the system with  the scaling vector field. As  application we  give two  surprisingly simple proofs for   small data global regularity results  non-isotropic   systems of wave  equations  in $\RRR^{1+3}$  with cubic semilinear nonlinearities.  We hope that the techniques presented here are  relevant for the more difficult and important  case   of biaxial  refraction in crystal optics.

\end{abstract}

\setcounter{tocdepth}{2}

\tableofcontents

 \section{Introduction}
Comparing with the rich literature in the study of   nonlinear wave propagation in isotropic media, see  for example  \cite{Joh81,JohKla84, Kla80,Kla82,Kla83,Kl85,Kla86,Klainerman15,KlaPon83,KlaSid96,PusSha13,Sideris, And19} as a short sample\footnote{The list would  be far longer if we  were also  to include    the  relevant  literature
 on  nonlinear field equations   and general relativity.},  the literature on   anisotropic systems of nonlinear waves is   far  sparser.   A notable  exception is    the  work  on global existence for the nonlinear equations
  of crystal optics\footnote{  Further results concerning Strichartz-type estimates in the context of  chrystal optics  can be found in \cite{Schippa}.}   due to  O. Liess \cite{Lie89, Lie91}.  Other  recent   results   are due to  \cite{And22}
and \cite{Wang23}, \Red{} As it is well known,  the proof of  long time    existence of solutions to nonlinear wave  equations
 requires   decay estimates for the  linearized system,   strong  enough  to control the nonlinear terms.  We have  two very  different   methods to derive decay, one based  on  Fourier and stationary phase methods, which provide $L^1 $ to $L^\infty$   estimates, and the   vectorfield methods,  based on commutation properties of the linearized system with   special    vectorfields,   which provide  $L^2$ to $L^\infty$  decay estimates. So far  only the  first  method   was  used for deriving decay estimates  for anisotropic  linear systems, see \cite{Lie91}. It was  widely believed that  vectorfield method  
 is inapplicable  to such systems due to the  lack of   enough  commuting vectorfields.  The goal of this  paper  is   demonstrate that this perception is wrong and that, at least, for   systems of   nonlinear,   linearly decoupled,   wave equations   of the form: 
 \begin{equation}\lab{eq:systems12}
\left\{\begin{array}{l}
 \, ^{(1)} \square\phi_1 = N_1\\
 \cdots\\ 
    \, ^{(m)} \square \phi_m = N_m\\ 
\end{array}
\right.
 \end{equation}
 where $m\in \mathbb{Z}_+, m\geq 2$ and $\forall i\in \{1,\cdots, m\}$, $ \, ^{(i)}\square:=  -  \pr_t^2 + \ep_1^i\pr_i^2+\ep_2^i \pr_2^2+\ep_3^i \pr_3^2$
 with positive  constants $ \ep_1^i, \ep_2^i, \ep_3^i$,
  the vectorfield method produces  realistic  decay estimates.  By a simple scaling  transformation  the system
   is equivalent  to the one in which   one of the two wave operators  is  the standard one in the Minkowski  space.
  The  key observation  is  that  both wave operators 
  $ \, ^{(1)} \square, \cdots,  \, ^{(m)} \square$  commute  with the scaling operator $S= t\pr_t +x^1\pr_1+x^2\pr_2+x^3 \pr_3$, i.e.
  \beaa
  \forall i\in\{1,\cdots, m\},\quad  [S,  \, ^{(i)} \square]= 2 \, ^{(i)} \square.
  \eeaa
  As a consequence of this  fact the problem of deriving combined decay estimates  for  all $\phi_1,\cdots, \phi_m$ of the linear system
   reduces to the problem of deriving  decay estimates  for  solutions of the standard wave equations $\square \phi=0$  which depend only on the commutation properties of  $\square$ with the standard derivatives $\pr_1, \pr_2, \pr_3$  and $S$. We provide two type of  results  one which establishes an $L^1-L^\infty$   estimate with $t^{-1} $ decay, based on  the stationary phase method,  and a second, restricted   to compactly supported  data,   that provides\footnote{We  give two   derivation for this second result, one based on a purely physical space argument   and a second one   which relies again on  the  method of stationary phase.} an  $L^2-L^\infty$    estimate with decay of the order of $t^{-\frac 12 } \big( 1+ \big|t-|x|\big|\big)^{-\frac 1 2 }$.      

    In the last section of the paper we apply  these estimates  to derive   preliminary  results   for  systems of type \eqref{eq:systems12} with cubic\footnote{Note that  cubic nonlinearities    are quite natural in  the context of nonlinear optics, see \cite{Metivier}. } nonlinearities $N_1, \cdots, N_m$  in the first derivatives of $\phi_1,\cdots, \phi_m$.

    \subsection{Connection  to  crystal optics}
    Systems of type \eqref{eq:systems12}   can be derived  from  the  following model of a 
 Maxwell system \eqref{classicform}  in homogeneous anisotropic media.
\be\label{classicform}
\textup{(Anisotropic Maxwell)}\quad \left\{\begin{array}{l}
\p_t B + \nabla\times E=0\\ 
\p_t D -\nabla \times H=0\\ 
\nabla\cdot D =0, \,\, \nabla\cdot B=0\\ 
D=\begin{bmatrix}
\epsilon_1 & 0 & 0\\ 
0 & \epsilon_2 & 0\\ 
0 & 0 & \epsilon_3\\
\end{bmatrix} E, \quad B= \mu H,
\end{array}\right.
\ee
where $D$ denotes the electric displacement, $E$ denotes the electric field, $H$ denotes the magnetic field, and $B$ denotes the magnetic induction. Moreover, $\epsilon_i \in \R_{+}, i \in\{1,2,3\},$ are called dielectric constants and $\mu$ is the tensor of magnetic permeability.  Without loss of generality,  one can  assume that  $\epsilon_1\geq    \epsilon_2 \geq \epsilon_3$ and $\epsilon_1=1.$ 

 The dielectric constants of crystals typically exhibit nonlinear behavior with respect to the electric field.    As mentioned above, the first step   in controlling the behavior of the nonlinear system is to 
 derive decay estimates  to  linear case, i.e when   $\ep_i$ ($i = 1, 2, 3$)  constants.
 
\begin{enumerate}
\item[(i)]
 With identical $\epsilon_1, \epsilon_2, \epsilon_3$, the system reduces to the classic Maxwell system, satisfying the standard second-order wave equation. This system admits a complete set of commuting rotational and Lorentz vector fields. A full analysis of   $L^2-L^\infty$ decay estimates   based on the vectorfield method  was given in \cite{CK90}.

 \item[(ii)] If only two of  the constants  $\epsilon_1, \epsilon_2, \epsilon_3$ are equal (the uniaxial case), the anisotropic Maxwell system becomes a coupled multi-speed wave system, with each wave satisfying a second-order wave equation. This system corresponds to a special case\footnote{ With $m=2$, in which $\,^{(1)}\square=\square $, $\ep_0=1$ and  two of the  $\ep$ constants coincide.} of  \eqref{eq:systems12} thus lacking   a complete set of commuting rotational and Lorentz vector fields. 

\item[(iii)] If all three parameters $\epsilon_1, \epsilon_2, \epsilon_3$ are distinct (the biaxial case), the anisotropic Maxwell system transforms into a coupled multi-speed wave system characterized by fourth-order wave equations for each wave component,  see \cite{Lie91}.  This is the case when the only commuting vectorfield is $S$. 

\end{enumerate}

\subsection{Main notation}

We consider the Minkowski space\footnote{ For simplicity of the presentation   we restrict our discussion to $n=3$, even though 
many  of the things done here can  be easily extended to  other dimensions.} $\RRR^{1+3}$  with  standard coordinates   $t=x^0, x^1, x^2, \ldots x^3$. We denote by 
$  \Si_t$  the time slices  of constant $t$ and $\Sext_t$ the exterior region of $\Si_t$ where $|x|\geq t$. Given  functions $\psi=\psi(t, x)$ we denote by       $\nab\psi=(\pr_1, \pr_2, \pr_3 )\psi$ its  spatial gradient and by $\pr \psi$ the spacetime gradient  $(\pr_t \psi, \nab \psi)$.  In  polar coordinate,  $r$ denotes  the radial function and $S(t,r)$ the sphere of radius $r$ on $\Si_t$. We  denote by $\nabb$ the induced covariant derivative on $S=S(t,r)$ 
  and  denote  by $\lapp=\de^{ab}\nabb_a\nabb_b $ the induced    Laplace-Beltrami operator on $S$.  Here  $(e_a)_{a=1,2}$ denotes  an arbitrary    orthonormal frames on $S$.   Given such a frame $e_1, e_2$  we note    that $\pr_r, e_1, e_2$  forms 
   an orthonormal frame on  $\Si_t$.  In polar coordinate the $3D $ Laplacian  can be  written in the form
    $\lap f = r^{-2} \pr_r(r^2 \pr_r f) +\lapp f$.  We  write  the wave operator in the form  $\square=-\p_t^2 +\Delta.$ 
    We denote by $\nab^k, \nab^\le k, \nabb^k, \nabb^{\le k}  $ the set of  the corresponding   partial   derivatives  of order $k$ and  respectively less or equal to $k$. 
     We introduce similar notation for  the set of all   commuting  vectorfields  $\Ga\in \big\{S, \pr_1, \pr_2, \pr_3\}$,  where  $S$ is the scaling vectorfield  $S=t\pr_t +x^i \pr_i$,  and denote by $ \Ga^{\le k} $ the set  of all $\Ga$  derivatives   of order    $\le k$.  Given a  function $\psi=\psi(t,x) $  the  corresponding $L^2$ norms   on the time slices $t$ are  written in the form  $\|\Ga^{\le k}\psi(t)\|_{L^2} $. 
     
For any two numbers $A$ and $B$, we use  $A\lesssim B$, $A\approx B$,  and $A\ll B$ to denote  $A\leq C B$, $|A-B|\leq c A$, and $A\leq c B$ respectively, where $C$ is an absolute constant and $c$ is a sufficiently small absolute constant. We use $A\sim B$ to denote the case when $A\lesssim B$ and $B\lesssim A$. 

 We use the convention that the time ``$t$'' presented in this paper is much larger than an absolute constant $C$, otherwise $L^\infty_x$ can be trivially bounded by the energy norm from the classic Sobolev embedding. 

     Let   {$\chi(x)$ be the  function on $\RRR$ defined as follows}
\be\label{2025July28eqn1}
 \chi(x):=\left\{\begin{array}{ll}
 0 & x\leq -3\\ 
 \frac{1}{4}(x+3)^{11} & x\in [-3,-2]\\ 
 \frac{1}{4}[-19(x+1)^{11}-22(x+1)^{10} +4]& x\in [-2,-1]\\  
 1& x\in [-1,0]\\ 
\chi(-x) & x\geq 0.
 \end{array}\right. 
 \ee
A straightforward computation shows that
\be\label{2025July28eqn2}
  \frac{|\p_x\chi(x)|^2}{\chi(x)} +  \frac{|\p_x^2\chi(x)|^2}{\chi(x)} \lesssim 1. 
\ee 
For any $k,k_1,k_2\in \Z$, we define the following rescaled bump function,
\be\label{2025July28eqn3}
\chi_{\leq k}(x):= \chi(2^{-k}x),\quad \chi_{k}(x):= \chi_{\leq k}(x)- \chi_{\leq k-1}(x), \quad \chi_{[k_1,k_2]}(x):=\sum_{k\in[k_1,k_2]\cap\Z} \chi_{k}(x).
\ee
  
    \subsection{Main results}

    As mention earlier the  goal of the paper is  to derive   realistic decay estimates  for solutions  to  standard  wave equation
    \bea
  \lab{eq:WEQ1}\lab{eq:scalarwave}
  \square \phi=f 
  \eea
  based only on the commutation properties of $\square $ with   the vectorfields  $\Ga\in\big\{\pr_1, \pr_2, \pr_3, S\big\}$, and then   use  these estimates to derive   global existence results for nonlinear  wave systems  of type \eqref{eq:systems12}.
  \ Given that  the scaling $S$ is only timelike  in the interior of the light cone $|x|<t$  in the first results given below,
     {Part (i)  of Theorem \ref{Thm:main1}},   we restrict our attention to solutions supported in that region, i.e. we assume that  that the initial data at $t=1$ has compact support  in the disk $|x|<1$ and that $f$ is likewise  supported in  $|x|<t$.  { In part (ii) of the Theorem, where we remove  the  compact support assumption, we get a  weaker  estimate in the exterior region, where $|x|>t$.   }
   
\begin{theorem}\lab{Thm:main1}
Given  $\delta>0$,  a  sufficiently small absolute constant, for solutions $\phi$  of \eqref{eq:scalarwave},  we have
\begin{enumerate}
  \item[(i)] If the solution is supported in the interior region, i.e., $|x|\leq t$, we have
\be\lab{finalestphysicspace}
\begin{split}
|\pr \phi(t,x)|&\lesssim   \delta^{-1}\Big[ \langle t \rangle^{-(1-\delta) /2 }\langle t -|x|\rangle^{-1/2+\delta }   \|\pr \,\Ga^{\le 3}\phi(t)\|_{L^2_x}   + \langle t\rangle^{1/2+\delta  }\| \Ga^{\le 2} f(t)\|_{L^2_x}\Big].
\end{split}
\ee
 \item[(ii)]  {In the general case, for any $x_0\in \R^3$}, the following estimate holds 
 \be\lab{finalestphysicspaceexterior}
 \begin{split}
 |\pr \phi(t,x_0)|   & \lesssim  
\delta^{-1} |t|^{-1/2} \langle x_0\rangle^{\delta} \|\Gamma^{\leq 3}\phi(t) \|_{L^2}^{(1-\delta)/2}\Big[    \big( \|\Gamma^{\leq 3}\phi(t)\|_{L^2} +  \|r^2\Gamma^{\leq 3}\phi(0)\|_{L^2}\big)   \\ 
 &\quad   + \langle t\rangle \| \Gamma^{\leq 1} f(t)\|_{L^2} +    \big(\int_0^t \|   u  \pr  \Gamma^{\leq 1} f\|_{L^2(\Sext_t)} \big)\Big]^{(1+\delta) /2} +\| \Gamma^{\leq 1} f(t)\|_{L^2}. 
 \end{split}
\ee
\end{enumerate}
  \end{theorem}
    \begin{remark}
  The proof  {of part (i) }  is based  on the fact that once we  control  the energy norm of  $S\phi$,
    the wave  equation  can be re-expressed, inside   the region where $S$ is time-like,
    as an inhomogeneous  elliptic  equation  for  $\phi$ which degenerates    along the null cone  $|x|=t$. 
    Using  a   B\"ochner identity  type calculation (Lemma \ref{WDBI})  this 
       allows us to get   $L^2$ control on the  second derivatives  of $\phi$. More precisely 
       we   get a very  good control for the  second angular derivatives $\nabb^2 \phi$  
         with degeneration   for the radial derivatives ($\partial_r^2 \phi$ and $\slashed{\nabla} \partial_r \phi$).
         These $L^2 $ estimates are then transformed  into  degenerate  sup-norm estimates  with the help of 
           a   weighted version  of the  Sobolev embedding Lemma  (see Lemma \ref{Sobolevembedding})
           tailored to take  full advantage of the  information given by our  degenerate  B\"ochner  identity (DBI).
     Due to the insufficient decay rates in the $L^2$ framework, the inhomogeneous term $f$ in \eqref{eq:scalarwave}  requires careful handling. To address this, we strategically distribute the angular and radial derivatives across different components in the weighted Sobolev embedding, allowing for improved control over the inhomogeneous term in \eqref{finalestphysicspace}. For further discussion, see Remark \ref{remark1}.  { In the proof of   part  (ii)  we observe that  negative  term in the same Bochner identity  can be controlled   in the exterior 
 region  by  an energy estimate   similar to one previously  used in     \cite{Klainerman16}. }
\end{remark} 
 {\begin{remark}
The estimate in \eqref{finalestphysicspace} is derived by applying the DBI twice. Repeated application of the DBI would further improve the decay rate to $\langle t -|x_0|\rangle^{-1+}$. This improvement stems from the fact that the most problematic term in the DBI (see \eqref{2025July25eqn1}) always involves a radial derivative. Iteratively applying the DBI to this term would, in principle, lead to a decay rate of $\langle t -|x_0|\rangle^{-1/2-1/4-1/8-\cdots}$, which simplifies to $\langle t -|x_0|\rangle^{-1+}$.  
 \end{remark}  }
 {Part (i)  of our   second theorem, based on the method of stationary phase,  { gives more precise  decay information}  than  Theorem \ref{Thm:main1}. Part (ii)   provides  a sharp  $L^1-L^\infty$  decay estimate  with a 
  $t^{-1}$ rate of decay.}

\begin{theorem}\lab{Thm:main2}
Considering the inhomogeneous wave equation $\square \phi=f$ with the initial data $(\phi,\p_t\phi)\big|_{t=0}=(\phi_0,\phi_1)$, we have the following two linear estimates. 
\begin{enumerate}
\item[(i)] We have\footnote{The  estimate  below shows that the ``$\delta$'' loss of time decay  in Theorem \ref{Thm:main1} can be removed,  using  the more precise method   stationary  phase. Moreover    the decay rate with respect to the light cone can also  be improved  in the interior region.} 
\be\label{2025July23eqn10}
|\pr \phi(t,x)|\lesssim  \delta^{-1}\Big[  \langle t \rangle^{-1 /2 }\big(\langle t -|x|\rangle^{-1+\delta   }\mathbf{1}_{|x|\leq t} + \mathbf{1}_{|x|> t}\big)  \|\pr \,\Ga^{\le 3}\phi(t)\|_{L^2_x}   + \langle t\rangle^{1/2  }\| \Ga^{\le 2} f(t)\|_{L^2_x}\Big] . 
\ee

\item[(ii)]   If $\phi$ is frequency localized, i.e.   $supp(\hat{\phi}(t,\xi))\subset \{\xi:\xi\in\R^3, |\xi|\sim 2^k\}$, for some  $  k\in \Z $, we have\footnote{ The  estimate   is tailored for applications to  the nonlinear problem.}   
\be\label{2025July23eqn11}
\begin{split}
|\pr \phi (t,x)|  &\lesssim  \langle t\rangle^{-1} 2^{2k}  \|    \widehat{\pr \Gamma^{\leq 2}\phi }(t,\xi)    \|_{L^\infty_\xi }      +  \langle t \rangle^{1/2}2^k \| \Gamma^{\leq 1} f (t,x)\|_{L^2}       .
 \end{split}
 \ee 
 Moreover, as a corollary of \eqref{2025July23eqn11}, for the homogeneous wave equation, i.e., $f=0$, we have\footnote{The  estimate  shows that,  for the standard  homogeneous case, the sharp uniform  decay rate is achievable by only using the scaling  vector field.} 
\[
  \|\pr \phi(t,x)\|_{L^\infty}\lesssim_{} \langle t\rangle^{-1} \big(\| \Gamma^{\leq 4} \phi_0 \|_{L^1_x} +\| \Gamma^{\leq 4} \phi_1 \|_{L^1_x}  \big).
\]

\end{enumerate}
\end{theorem}

\begin{remark}\label{geometricmeasure}
The proof of   {both parts} of   Theorem \ref{Thm:main2} relies on a crucial geometric observation regarding the localized Fourier space  set $S_{k,l}(t,x):=\big\{\xi: | 1+ {\mu x\cdot\xi}\big/{t|\xi|} |\leq  2^{l}, |\xi|\in  [2^{k-1}, 2^{k}], \mu\in\{+,-\} \big\}$, where $k,l\in \Z$.  This set emerges naturally from the classical stationary phase method. Remarkably, regardless of  $x\in \R^3, t\in \R/\{0\}$, the directions of  $S_{k,l}(t,x)$,  $\{\xi/|\xi|: \xi\in S_{k,l}(t,x)\}$,  are confined to a thin spherical segment of size $2^l$. This leads to a sharp measure estimate for $S_{k,l}(t,x)$, which in turn leads to the linear estimates  \eqref{2025July23eqn10} and \eqref{2025July23eqn11}. 
\end{remark} 

 {\begin{remark}
 The improved decay rate observed near the light cone in the interior region (cf. \eqref{finalestphysicspace} and \eqref{2025July23eqn10}) can be attributed to the ellipticity of the operator $\Delta-r^2 t^{-2}\p_r^2$ in that domain. This ellipticity is lost in the exterior region, which accounts for the lack of a corresponding estimate in \eqref{2025July23eqn10}.
\end{remark}}

 The following   results are applications  of these decay estimates   to global and long time existence results for the system of type   \eqref{eq:systems12} with cubic or higher nonlinearities\footnote{  {As mentioned  earlier,  cubic (Kerr) nonlinearities  appear natural in nonlinear optics,   see for example \cite{Metivier}[Section 2].}}.
   More precisely, we assume that the nolinearities $N_i, i\in\{1,\cdots, m\}$, are given as follows,
\be
N_i:=\sum_{a_1, a_2,a_3\in\{1,\cdots, m\}} C^i_{a_1,a_2,a_3}(\pr \phi_{a_1}, \pr\phi_{a_2}, \pr\phi_{a_3}),  
\ee
where  $ C^i_{a_1,a_2,a_3}(\cdot, \cdot, \cdot), i \in\{1,\cdots, m\}, a_1,a_2, a_3\in \{1,\cdots, m\}$ are  some  smooth functions.  
The following theorem establishes small-data global existence for the system \eqref{eq:systems12}    { using exclusively the generalized  energy norm (based on the vectorfields $\Ga\in\big\{ \pr_1, \pr_2, \pr_3, S\big\}$}.
\begin{theorem}\lab{Thm:main3}
Let $m\geq 3$. Assume that the initial data  $(\phi_i,\p_t \phi_i)\big|_{t=0}$, $i\in\{1,\cdots, m\}$, of  the coupled system  \eqref{eq:systems12} with nonlinearities as in \eqref{nonlinearityass}     are  regular and compactly supported.  Moreover, we impose the following  {three structural assumptions\footnote{These are technical assumptions  meant to take full advantage  of the estimates derived in  Theorems \ref{Thm:main1}, \ref{Thm:main2}. } of on our system}, 
\begin{enumerate}
\item[(i)] Separation assumption on the  light cones.\\ 
Define  $r_i:=\sqrt{ (\epsilon_{1}^{i})^{-2} x_1^2 + (\epsilon_{2}^{i})^{-2} x_2^2+ (\epsilon_{3}^{i})^{-2} x_3^2 }$. The light cones $t=r_i$ associated with the wave $\phi_i$ are well separated.  More precisely, for any $i, j\in \{1,\cdots, m\}$, we have
\be\label{2025June14eqn41}
  |r_i-r_j|\sim |r_i|\sim |r_j|. 
\ee
\item[(ii)] Cubic and higher order assumption for the nonlinearities. 

For $\forall i\in \{1,\cdots,m\},|\pr\phi_i|\leq 1$, we have
\be\label{2025July30eqn1}
\forall  i,a_1,a_2,a_3\in \{1,\cdots, m\},\quad  |C^i_{a_1,a_2,a_3}(\pr \phi_{a_1}, \pr\phi_{a_2}, \pr\phi_{a_3})|\lesssim \sum_{a\in \{1,\cdots,m\}} |\pr \phi_a|^3.
\ee
\item[(iii)] No self-interaction assumption.\\ 
For any $i,a_1,a_2,a_3\in \{1,\cdots, m\}$, we assume that   $C^i_{a_1,a_2,a_3}(\cdot, \cdot, \cdot)=0$ if there exist distinct indices $i, j \in \{1, 2, 3\}$, such that $a_i = a_j$.
\end{enumerate}
Then there exists a sufficiently small absolute constant $\epsilon_0>0$ s.t., if  the following estimate holds for the initial data
 holds, 
\be\label{nonlinearityass}
    \sum_{
\begin{subarray}{c}
i\in \{1,\cdots, m\}   
\end{subarray}} \|\pr\Gamma^{\leq 30}\phi_i\|_{L^2_x}\big|_{t=0}  \leq \epsilon_0,
\ee
then the coupled system \eqref{eq:systems12}  has a unique global solution. Moreover, for any $i\in\{1,\cdots,m\}$, the nonlinear wave  $\phi_i$ scatters to a linear wave. 
  \end{theorem}

  {\begin{remark}
The proof  of Theorem \ref{Thm:main3}   relies on the decay estimate \eqref{finalestphysicspace} established in Theorem \ref{Thm:main1}.  {Let $k_0=30$ denote the highest order derivative used in the energy estimate}. The separation assumption on the light cone, combined with the no self-interaction assumption, guarantees that at least two of the three inputs in  the nonlinearities $N_i$ exhibit an enhanced decay rate of $\langle t\rangle^{-1+}$ due to the additional decay along the light cones. For example\footnote{ In view of  the separation condition in \eqref{2025June14eqn41}.}, given a cubic  nonlinearity $C(\phi_1, \phi_2, \phi_3)$, according  to   \eqref{finalestphysicspace},  the decay rates of $\phi_1 $ and $\phi_2$ are improved to  $\langle t\rangle^{-1+}$,  near the light cone of $\phi_3$.

For lower-order estimates (e.g., $k\leq k_0-5$), we obtain an improved decay rate of $\langle t \rangle^{-2+}$ for $\|\Gamma^{\leq k} N_i\|_{L^2}$. This improvement stems from the flexibility to choose which two inputs are estimated in $L^\infty$. Furthermore, for higher-order estimates (e.g., $k_0-4\leq k \leq k_0$), $\|\Gamma^{\leq k} N_i\|_{L^2}$ decays at a rate of $\langle t \rangle^{-3/2+}$ because at least one input in $L^\infty_x$ benefits from the enhanced $\langle t \rangle^{-1+}$ decay. This crucial hierarchical structure, with its varying orders of $|\Gamma^{\leq k} N_i|_{L^2}$, prevents the accumulation of $t^\delta$ losses, thereby stabilizing the bootstrap argument.
\end{remark}}

Our last  theorem establishes that small-data global existence for the system \eqref{eq:systems12} with \emph{ cubic or higher order nonlinearties} (as in \eqref{2025July30eqn1}) holds true, provided the initial data belongs to $L^1_x$, and no other conditions are necessary.
  \begin{theorem}\lab{Thm:main4}
     Let $m\geq 2$. There exists a sufficiently small absolute constant $\epsilon_0>0$ s.t., if  the following estimate holds for the initial data
 holds, 
\[
     \sum_{
\begin{subarray}{c}
i\in \{1,\cdots, m\}   
\end{subarray}} \|\pr\Gamma^{\leq 30}\phi_i\|_{L^2_x\cap L^1_x}\big|_{t=0}  \leq \epsilon_0,
\]
then the coupled system  \eqref{eq:systems12} has a unique global solution. Moreover, for any $i\in\{1,\cdots,m\}$, the nonlinear wave  $\phi_i$ scatters to a linear wave. 
   \end{theorem}  

  \begin{remark}
The proof of Theorem \ref{Thm:main4}  relies on the decay estimate \eqref{2025July23eqn11} established in Theorem \ref{Thm:main2}. We simultaneously control the energy and the $L^\infty_\xi$-norm of the solution. Briefly, the Duhamel formula for the wave equation suggests that controlling the $L^\infty_\xi$-norm requires controlling the $L^1_x$-norm of the nonlinearity $N_i$. This can be achieved using an $L^2_x-L^2_x-L^\infty_x$ type estimate. However, even with the sharp $\langle t \rangle^{-1}$ decay, a logarithmic loss arises, hindering a standard bootstrap argument. To circumvent this, we permit a slow growth of the $L^\infty_\xi$-norm, specifically $\langle t\rangle^{\delta}$. Consequently, the decay rate becomes $\langle t \rangle^{-1+\delta}$, as per the estimate \eqref{2025July23eqn11} in Theorem \ref{Thm:main3}. Crucially, because the nonlinearities $N_i$ are cubic, the $L^2$-estimate benefits from a decay rate of $\langle t \rangle^{-2+2\delta}$, which is sufficient for the energy estimate. Furthermore, this $\langle t \rangle^{-1+\delta}$ decay rate of the $L^\infty_x$-norm ensures that the $t^\delta$ growth rate remains consistent with the $L^2_x-L^2_x-L^\infty_x$ estimate when controlling the $L^1_x$-norm of the nonlinearities $N_i$. 
\end{remark}

\begin{remark}
We note that though  the result of Theorem \ref{Thm:main4}  is  stronger   than that  of Theorem  \ref{Thm:main3} we  believe however that the   physical space methods used in the proof of Theorem  \ref{Thm:main3}  are of independent interest, of possible  relevance to  quasilinear equations.
\end{remark}
 
  We hope that the techniques presented in this paper  are  relevant for the more difficult and important  case   of biaxial  refraction in crystal optics.

 \vo 

\noindent \textbf{Acknowledgement}

S.K acknowledges support from  the NSF grant DMS220103. X.W acknowledges support from NSFC-12322110, 12141102, 12326602, and MOST-2024YFA101500, 2020YFA0713003. {We would like to  thank J. Anderson
 for his interest in this work and  for many very useful comments.}

  \section{ Proof of Theorem \ref{Thm:main1}.   Decay estimates for the standard wave equation using only  commutations  with the scaling operator}
 
  The proof of the theorem is based on the following  ingredients.

  \subsection{Decomposition of the wave operator   using  $S$ and  main energy estimates}
  
  As mentioned in the introduction we plan to use the scaling operator  to study the decay properties of the equation \eqref{eq:WEQ1}, i.e.$ \square \phi=f$,
  Since the scaling $S$ is only timelike  in the interior of the light cone $|x|<t$ we restrict our attention to solutions supported in that region, i.e. we assume that  that the initial data at $t=1$ has compact support  in the disk $|x|<1$ and that $f$ is likewise  supported in  $|x|<t$.
     \begin{lemma}
The  wave equation $\square\phi=f$ can be written in the form
\bea
\lab{eq:main-formula}
 \lapp \phi  + \big( 1-\frac{r^2}{t^2} \big) r^{-2} \pr_r^2 (r^22 \phi) =    \frac{2t\pr_t S\phi -S^2\phi  - S \phi}{t^2} +\square \phi
\eea
\end{lemma}
  \begin{proof}
  We start with the observation  that  $\pr_t \phi = \frac{S\phi - x\cdot \nabla \phi}{t},$   and therefore
\be\label{2025July28eqn21}
\begin{split}
  \pr_t^2 \phi &= t^{-1}\Big[\pr_t S \phi- x \c \nab \pr_t \phi \Big]- t^{-2}\big(S \phi- x \c \nab \phi \big)\\
  &= t^{-1}\Big[   t^{-1}\big( S^2 \phi- x \c \nab  S  \phi \big)-\big(   t^{-1} (x \c \nab ) S\phi  - t^{-1} \big(x\c \nab \big)^2 \phi\big)\Big] - t^{-2}\big(S \phi- x \c \nab \phi \big)\\
  &=   t^{-2} (x\c\nab)^2\phi + t^{-2} (x\c \nab)\phi +t^{-2} \Big( S^2\phi - 2(x\c \nab) S\phi - S\phi\Big)\\
  &=  t^{-2} (x\c\nab)^2\phi + t^{-2} (x\c \nab)\phi +t^{-2}\Big(- S^2 \phi + 2 t \pr_t S\phi- S\phi\Big).
\end{split}
\ee 
Since $x\c\nab =r\pr_r$ and  $(x\cdot \nabla)^2\phi =( r\pr_r )^2\phi = r^2 \pr_r^2 \phi +r \pr_r \phi$
 we deduce
\beaa
\lap \phi&=&(\pr_t^2+\square) \phi=  \frac{r^2}{t^2} \pr_r^2 \phi  +   2  \frac{r}{t^2} \pr_r \phi+   \frac{2t\pr_t S\phi -S^2\phi  - S \phi}{t^2} +\square \phi.
\eeaa
 and, since $\lap= \lapp+\pr_r^2 + \frac{2}{r} \pr_r$,
 \beaa
 \lapp \phi  + \big( 1-\frac{r^2}{t^2} \big)\pr_r^2\phi+ \frac{2}{r} \big( 1-\frac{r^2}{t^2} \big)\pr_r\phi=   
  \frac{2t\pr_t S\phi -S^2\phi  - S \phi}{t^2} +\square \phi.
 \eeaa
  from which \eqref{eq:main-formula} follws.
\end{proof}
 
We are led to study the equation
\bea
\lab{eq:wavepsi}
L[\phi] &=&F[\phi]+ f, \qquad  L=\lapp +  \big(1-     \frac{r^2}{t^2} \big) r^{-2} \pr_r(r^2 \pr_r)
\eea
with
\be 
\lab{eq:waveps-F1}
\begin{split}
F[\phi] 
&= \frac{1}{t^2} \big( 2t\pr_t S\phi -S^2\phi  - S \phi \big)=\frac{1}{t^2} \big( S^2\phi -2r\p_r S\phi- S \phi  \big).
\end{split}
\ee

\subsubsection{Energy estimates}
By standard energy estimates  we can easily show that solutions of  \eqref{eq:WEQ1} verify the energy estimates
\beaa
\|\pr\phi(t)\|_{H^s} &\les & \|\pr\phi(0)\|_{H^s} +\int_0^t \|f(\tau)\|_{H^s}
\eeaa
Using   the fact that $S$ commutes with $\square$, i.e. $[\square, S]=-2\square $  we also have
\beaa
\|\pr S\phi(t)\|_{H^{s-1}} &\les & \|\pr S\phi(0)\|_{H^s} +\int_0^t \Big(\| S f(\tau)\|_{H^{s-1}}+ \|f(\tau)\|_{H^s}\Big) \\
|\pr S^2\phi(t)\|_{H^{s-2}} &\les & \|\pr S^2\phi(0)\|_{H^s} +\int_0^t \Big(\| S^2 f(\tau)\|_{H^{s-2}} +\| S f(\tau)\|_{H^{s-1}}+ \|f(\tau)\|_{H^s}
\eeaa
We introduce the notation $\Ga^{\le k}\phi $ to   the set  of all $\Ga$  derivatives   of order    $\le k$ where  $\Ga$ 
denotes any vectorfield in the set  $ \big\{S, \pr_1, \pr_2, \pr_3\}$. With these notation  we  summarize the result as follows.
 \begin{lemma}
 \lab{Le:energy-conserv}
 The solutions of the wave  equation  $\square \psi=f$ verify,  for all $k\ge 0$,
 \bea\label{energyeestmain}
\|\pr  \Ga^{\le k}\phi(t)\|_{H^{s-k}}&\le & \|\pr \Ga^{\le k}\phi(0)\|_{H^{s-k}}+\int_0^t   \|\Ga^{\le k} f\|_{H^{s-k}}
\eea
 \end{lemma}
 As a consequence  of the Lemma we have for $F[\phi]$ given  by \eqref{eq:waveps-F1}
\bea
\|F[\phi]\|_{L^2(\Si_t)} \les  t^{-1}   \|\pr \Ga^{\le 2}\phi(t)\|_{L^2(\Si_t)}\les t^{-1}   \big[\|\pr \Ga^{\le 2}\phi(0)\|_{L^2(\Si_t)}+\int_0^t   \|\Ga^{\le 2} f\|_{L^2(\Si_t)}\big].
\eea

\subsection{  Degenerate    $L^2$ estimates for $L$}
  \lab{sec:Deg-elliptic}
Recall that, see \eqref{eq:wavepsi}, $L=\lapp +  \big(1-     \frac{r^2}{t^2} \big) r^{-2} \pr_r(r^2 \pr_r)$ is a degenerate elliptic operator in the region $|x|<t$.

\subsection{B\"ochner identities}

We derive below  first and second order   $L^2$  estimates for it. 
 \begin{lemma}
   \lab{Le:2nd-pointwise}
   We have the following point-wise identity.
\be\label{eq:2nd-pointwise}
\begin{split}
 \big|L[\phi]\big|^2 &=\big| \nabb^2 \phi\big|^2+    {2}\big(1-     \frac{r^2}{t^2} \big)\big|\nabb\pr_r\phi|^2+\Big|\big(1-     \frac{r^2}{t^2} \big) r^{-2} \pr_r( r^2\pr_r\phi)\Big|^2  {-   2r^{-2}}|\nabb\phi|^2+\Div[\phi]\,\,\end{split}
\ee
where $\Div[\phi]$ is the  pure divergence expression
\be\label{2025Aug4eqn11}
\begin{split}
  \Div[\phi] & =   2r^{-2}\p_r \big(\lapp \phi    \big(1-     \frac{r^2}{t^2} \big) r^2 \pr_r  \phi - r|\slashed \nabla \phi|^2 \big)  +      \nabb_a\Big( \nabb^a\phi \lapp \phi-\frac 1 2 \nabb^a(|\nabb\phi|^2)\Big)   \\ 
  &\quad  - 2\slashed\nabla_a\big(  \big(1-     \frac{r^2}{t^2} \big) \nab^a\pr_r\phi \pr_r \phi - 2 r^{-1}\slashed\nabla^a \phi \p_r \phi \big). 
\end{split}
\ee
  
   \end{lemma} 
  
  \begin{proof}
  We write
 \be\label{2025Aug4eqn2}
  \begin{split}
  \big|L[\phi]\big|^2 &= \Big| \lapp \phi+\big(1-     \frac{r^2}{t^2} \big) r^{-2} \pr_r( r^2\pr_r\phi)\Big|^2 \\
  &= \big| \lapp \phi\big|^2 +
  \Big|\big(1-     \frac{r^2}{t^2} \big) r^{-2} \pr_r( r^2\pr_r\phi)\Big|^2 +  {2} \lapp \phi  \big(1-     \frac{r^2}{t^2} \big) r^{-2} \pr_r( r^2\pr_r\phi).
  \end{split}
  \ee
  We first note
  \bea\lab{eq:Boch-lapp}
    |\lapp\phi|^2&=&    |\nabb^2 \phi|^2 + r^{-2}|\nab\phi|^2 +   \nabb_a\Big( \nabb^a\phi \lapp \phi-\frac 1 2 \nabb^a(|\nabb\phi|^2)\Big).
    \eea

  This  follows easily by by the usual commutation formula for covariant derivatives involving the Gauss curvature $K=r^{-2}$ of the spheres $S_r$. I particular we have  $[\nabb_a, \lapp ]\phi=-r^{-2} \nabb_a \phi$. 
  We next consider the mixed  term
  \beaa
  \lapp \phi  \big(1-     \frac{r^2}{t^2} \big) r^{-2} \pr_r( r^2\pr_r\phi)= r^{-2} I
  \eeaa
  where 
  \be
  \begin{split}
  I=&  \lapp \phi  \big(1-     \frac{r^2}{t^2} \big)   \pr_r( r^2\pr_r\phi)\\ 
  &=\pr_r\Big( \lapp \phi    \big(1-     \frac{r^2}{t^2} \big) r^2 \pr_r  \phi\Big)-\pr_r ( \lapp \phi )\big(1-     \frac{r^2}{t^2} \big)   r^2 \pr_r\phi-\lapp\phi \pr_r \big(1-     \frac{r^2}{t^2} \big)r^2\pr_r \phi. \\ 
  \end{split}
  \ee
 
  Making  use of the  commutation formulas   $[\nabb_r, \nabb]=-r^{-1} \nabb$  and $[\pr_r, \lapp]=- 2 r^{-1} \lapp$  
  \be
  \begin{split}
  &-\pr_r ( \lapp \phi )\big(1-     \frac{r^2}{t^2} \big)   r^2 \pr_r\phi\\ 
  &= -\big( \lapp\pr_r  \phi-\frac{2}{r}\lapp \phi\big)\big(1-     \frac{r^2}{t^2} \big)   r^2 \pr_r\phi\\
   &=-r^2 \big(1-     \frac{r^2}{t^2} \big)\nabb_a\big(\nabb^a\pr_r\phi \pr_r \phi\big)+r^2 \big(1-     \frac{r^2}{t^2} \big)\big|\nabb\pr_r\phi|^2 + \frac{2}{r}\lapp \phi\big(1-     \frac{r^2}{t^2} \big)   r^2 \pr_r\phi. 
   \end{split}
  \ee
  Hence
  \be\label{2025Aug4eqn1}
  \begin{split}
   I =&  \pr_r\Big( \lapp \phi    \big(1-     \frac{r^2}{t^2} \big) r^2 \pr_r  \phi\Big)+r^2 \big(1-     \frac{r^2}{t^2} \big)\big|\nabb\pr_r\phi|^2 -r^2 \big(1-     \frac{r^2}{t^2} \big)\nabb_a\big(\nab^a\pr_r\phi \pr_r \phi\big)\\
    &+ {\frac{2r^3}{t^2}} \lapp\phi \pr_r \phi + \frac{2}{r}\lapp \phi\big(1-     \frac{r^2}{t^2} \big)   r^2 \pr_r\phi\\
   &= r^2 \big(1-     \frac{r^2}{t^2} \big)\big|\nabb\pr_r\phi|^2 + 2 r\lapp \phi\pr_r\phi+ \pr_r\Big( \lapp \phi    \big(1-     \frac{r^2}{t^2} \big) r^2 \pr_r  \phi\Big) -r^2 \big(1-     \frac{r^2}{t^2} \big)\nabb_a\big(\nab^a\pr_r\phi \pr_r \phi\big)\\
   \end{split}
  \ee 
  We further write, with the help of the commutation formula      $[\nabb_r, \nabb]=-r^{-1} \nabb$, 
\[
 \begin{split}
    2r \lapp \phi\pr_r\phi & = 2r\slashed\nabla_a\big( \slashed\nabla^a \phi \p_r \phi \big)- 2r \big( \slashed\nabla^a \phi  \slashed\nabla_a\p_r \phi \big)\\ 
    &= 2r\slashed\nabla_a\big( \slashed\nabla^a \phi \p_r \phi \big)- 2r \big( \slashed\nabla^a \phi  \p_r  \slashed\nabla_a \phi  + r^{-1} |  \slashed\nabla \phi|^2\big)\\ 
    &= \slashed\nabla_a\big( 2 r\slashed\nabla^a \phi \p_r \phi \big) -  \p_r \big( r |\slashed\nabla \phi |^2 \big) -  |\slashed\nabla \phi |^2
  \end{split}
  \]
  Hence 
  \[
  \begin{split}
    I&=r^2 \big(1-     \frac{r^2}{t^2} \big)\big|\nabb\pr_r\phi|^2  + \slashed\nabla_a\big( 2 r\slashed\nabla^a \phi \p_r \phi \big) -  \p_r \big( r |\slashed\nabla \phi |^2 \big) -  |\slashed\nabla \phi |^2\\ 
    &+ \pr_r\Big( \lapp \phi    \big(1-     \frac{r^2}{t^2} \big) r^2 \pr_r  \phi\Big) -r^2 \big(1-     \frac{r^2}{t^2} \big)\nabb_a\big(\nab^a\pr_r\phi \pr_r \phi\big)\\
    &= r^2 \big(1-     \frac{r^2}{t^2} \big)\big|\nabb\pr_r\phi|^2-  |\slashed\nabla \phi |^2 +\p_r \big(\lapp \phi    \big(1-     \frac{r^2}{t^2} \big) r^2 \pr_r  \phi - r|\slashed \nabla \phi|^2 \big)\\ 
    &\quad + \slashed\nabla_a\big( 2 r\slashed\nabla^a \phi \p_r \phi - r^2 \big(1-     \frac{r^2}{t^2} \big) \nab^a\pr_r\phi \pr_r \phi \big). 
    \end{split}
  \]
Thus back to the identity
  \beaa
  \big|L[\phi]\big|^2 
  &=&\big| \lapp \phi\big|^2 +
  \Big|\big(1-     \frac{r^2}{t^2} \big) r^{-2} \pr_r( r^2\pr_r\phi)\Big|^2 + 2 \lapp \phi  \big(1-     \frac{r^2}{t^2} \big) r^{-2} \pr_r( r^2\pr_r\phi),
  \eeaa
 recalling that   $\lapp \phi  \big(1-     \frac{r^2}{t^2} \big) r^{-2} \pr_r( r^2\pr_r\phi)= r^{-2} I$ and         using also \eqref{eq:Boch-lapp}
  we deduce
  \beaa
  \big|L[\phi]\big|^2 
  &=&\big| \nabb^2 \phi\big|^2+\Big|\big(1-     \frac{r^2}{t^2} \big) r^{-2} \pr_r( r^2\pr_r\phi)\Big|^2+  2 \big(1-     \frac{r^2}{t^2} \big)\big|\nabb\pr_r\phi|^2 -\frac{2}{r^2}|\nabb\phi|^2+\Div[\phi]\\
  \Div[\phi]&=&  2 r^{-2} \pr_r\Big( \lapp \phi    \big(1-     \frac{r^2}{t^2} \big) r^2 \pr_r  \phi- r|\nabb\phi|^2 \Big) +  \nabb_a\Big( \nabb^a \phi \lapp \phi-\frac 1 2 \nabb^a(|\nabb\phi|^2)\Big)\\
  &&-   2\nabb_a  \Bigg(   \big(1-     \frac{r^2}{t^2} \big)  \nab^a\pr_r\phi \pr_r \phi- 2\frac{1}{r} \nabb^a\phi \pr_r\phi\Big) 
  \eeaa
  as stated.
    \end{proof}

For the sake of completeness  we also record below  the  first derivative estimates.

  \begin{lemma}
  \lab{Le:1st-pointwise}
  The following point-wise identity holds  true
  \bea
  \lab{eq:1st-pointwise}
  \bsplit
   -L[\phi]\phi &=|\nabb \phi|^2 + \big(1-     \frac{r^2}{t^2} \big) |\pr_r\phi|^2 +\frac{3}{t^2}|\phi|^2 -\Div_1[\phi]\\
   \Div_1[\phi]&=\nabb_a\big(\phi \nabb^a\phi\big) + r^{-2} \pr_r \Big( \big(1-     \frac{r^2}{t^2} \big)  \phi  r^2\pr_r\phi +\frac{r^3}{t^2}|\phi|^2 \Big)
   \end{split}
  \eea
  \end{lemma}
  \begin{proof}
 Straightforward verification.
   \end{proof}

  \subsection{Main Integral identities}
 As a consequence of the  Lemma \ref{Le:2nd-pointwise} 
    we derive the  degenerate  B\"ochner identity
  \begin{lemma}\label{degenerateBochner}
  \lab{Le:Bochner}
  Solutions to  the equation   $\square \phi=f $,   verify 
  \bea
 \lab{eq:Bochner2}
  \int_{\Si_t} \big| \nabb^2 \phi\big|^2+ 2  \big(1-     \frac{r^2}{t^2} \big)\big|\nabb\pr_r\phi|^2 &\les &
  \int_{\Si_t}  (t^{-2}+r^{-2}  )   \big|\pr \Ga^{\le 1 } \phi \big|^2 + \int_{\Si_t} |f|^2
  \eea
  \end{lemma}
 \begin{proof}
  To check \eqref{eq:Bochner2}  we integrate  \eqref{2025Aug4eqn11} on $\Si_t$ and derive the identity
\bea\label{2025Aug4eqn61}
 \int_{\Si_t} \big| \nabb^2 \phi\big|^2+  2 \big(1-     \frac{r^2}{t^2} \big)\big|\nabb\pr_r\phi|^2 &\les & \int_{\Si_t} \big|L[\phi]\big|^2 +   \int_{\Si_t} r^{-2} |\nabb\phi|^2     .
  \eea
  To estimate  the right hand side we  recall that $ L[\phi]=F[\phi]+ f$  with $F[\phi]= \frac{1}{t^2} \big(2t\p_t S\phi - S^2\phi - S \phi  \big)$, see \eqref{eq:waveps-F1}. Thus
  \beaa
  \int_{\Si_t} \big| L[\phi]\big|^2 &\les &   \int_{\Si_t}    (t^{-2}+r^{-2}  )   \big|\pr \Ga^{\le 1 } \phi \big|^2 + \int_{\Si_t} |f|^2.
  \eeaa
  This ends the  verification  of \eqref{eq:Bochner2}
  \end{proof}
  \begin{remark}
  Note that the statement of the lemma remains valid even if  we drop the assumption that
  $\phi$ is supported in the region $|x|<t$.
  \end{remark}
 
 The estimate \eqref{eq:Bochner2} provides a good control of the second angular derivative of $\phi$ in interior region $|x|<t$ under the assumption that $\phi$ is supported  in $|x|<t$. In what follows   we show that similar identities    can be derived even if we remove that support assumption.
 
 \begin{lemma}\label{exteriorestimate}
 \lab{Le:Bochner-ext}
 The following estimate holds for general  solutions of $\square \phi=f$   in the  exterior region $|x|\ge t$.
   \be\label{2025Aug4eqn81}
 \begin{split}
    \int_{\Sext_t} \big| \nabb^2 \phi\big|^2  +   \big( \frac{r^2}{t^2}-1 \big)\big|\nabb\pr_r\phi|^2 &\les  t^{-2} \int_{\Si_0}  r^2  |\pr^2\phi|^2 + t^{-2}  \big(\int_0^t \|   u  \pr \square \phi\|_{L^2(\Sext_t)} \big)^2
 \\
 &+ t^{-2}  \int_{\Sext_t}   \big|\pr \Ga^{\le 1 } \phi \big|^2 + \int_{\Sext_t} |f|^2.
 \end{split}
 \ee
 \end{lemma}

 \begin{proof}
 We  integrate the identity  \eqref{2025Aug4eqn11}   on $\Sext_t$. Observe that
  \beaa
 \int_{\Sext_t} \Div[\phi]=\int_t^\infty   \int_{\SSS^2}2 \pr_r\Big( \lapp \phi    \big(1-     \frac{r^2}{t^2} \big) r^2 \pr_r  \phi- {r} |\nabb\phi|^2 \Big)= 2 t^{-1} \int_{r=t} | \nabb\phi|^2 
 \eeaa
 Therefore,  after neglecting the nonnegative term $\Big|\big(1-     \frac{r^2}{t^2} \big) r^{-2} \pr_r( r^2\pr_r\phi)\Big|^2 $, we have 
 \bea
 \lab{eq:Ext0}
 \int_{\Sext_t} \big| \nabb^2 \phi\big|^2+  2 \big(1-     \frac{r^2}{t^2} \big)\big|\nabb\pr_r\phi|^2  -\frac{2}{r^2}|\nabb\phi|^2+ 2 t^{-1} \int_{r=t} | \nabb\phi|^2  & \leq   &  \int_{\Sext_t} \big|L[\phi]\big|^2
 \eea
 Hence
 \bea
 \lab{eq:Ext1}
 \begin{split}
 \int_{\Sext_t} \big| \nabb^2 \phi\big|^2 -\frac{2}{r^2}|\nabb\phi|^2 + 2 t^{-1} \int_{r=t} | \nabb\phi|^2  & \leq  2t^{-2}  \int_{\Sext_t} \big(r^2-t^2\big)\big|\nabb\pr_r\phi|^2\\
 &+  \int_{\Sext_t} \big|L[\phi]\big|^2.
 \end{split}
 \eea
 Integrating by parts  we  derive in the integral $I=\int_{\Sext_t} (r^2-t^2) \big|\nabb\pr_r\phi|^2$ we deduce
 \beaa
 I&=& \int_{\Sext_t} (r^2-t^2) \big|\nabb\pr_r\phi|^2=  \int_{\Sext_t} (r^2-t^2) \nabb\pr_r\phi\nabb\pr_r \phi= -\int_{\Sext_t} (r^2-t^2)  \pr_r\phi\lapp\pr_r \phi\\
  &=& -\int_{\Sext_t} (r^2-t^2)  \pr_r\phi\big(\pr_r \lapp\phi+ 2r^{-1} \lapp\phi\big)
  \\
  &=&   \int_{\Sext_t}  r^{-2} \pr_r \big( r^2(r^2-t^2)\pr_r \phi \big) \lapp \phi- \int_{\Sext_t} (r^2-t^2)\pr_r \phi 2r^{-1} \lapp\phi\\
  &=&   \int_{\Sext_t}  (r^2-t^2)\pr^2_r \phi \lapp \phi+ \int_{\Sext_t} \big( 4 r - 2 r^{-1} t^2 \big)\pr_r \phi \lapp \phi
  - \big( 2 r - 2 r^{-1} t^2 \big) \pr_r \phi \lapp \phi
  \eeaa
  Hence
  \beaa
 I &=& \int_{\Sext_t}  (r^2-t^2)\pr^2_r \phi \lapp \phi+ \int_{\Sext_t} 2 r  \pr_r \phi \lapp \phi\\
  &=& \int_{\Sext_t}  (r^2-t^2)\pr^2_r \phi \lapp \phi-  \int_{\Sext_t} 2 r  \nabb \pr_r \phi \nabb \phi\\
  &=& \int_{\Sext_t}  (r^2-t^2)\pr^2_r \phi \lapp \phi -  \int_{\Sext_t} 2 r \big(  \pr_r  \nabb\phi + r^{-1} \nabb\phi \big)\nabb \phi\\
  &=&  \int_{\Sext_t}  (r^2-t^2)\pr^2_r \phi \lapp \phi-   \int_{\Sext_t}  r \pr_r\big(|\nabb\phi|^2 \big) 
  -2 \int_{\Sext_t}|\nabb\phi|^2 
 \eeaa
 Now
 \beaa
  \int_{\Sext_t}  r \pr_r\big(|\nabb\phi|^2 \big) &=&\int_t^\infty \int_{|\xi|=1 } r  \pr_r\big(|\nabb\phi|^2 \big) r^2 d\si(\xi)=-  3   \int_{\Sext_t}  |\nabb\phi|^2 - t  \int_{|x|=t} |\nabb\phi|^2 
 \eeaa
 Hence
 \be\label{2025Aug4eqn31}
 I =  \int_{\Sext_t}  (r^2-t^2)\pr^2_r \phi \lapp \phi+\int_{\Sext}|\nabb\phi|^2 +   t  \int_{|x|=t} |\nabb\phi|^2 
 \ee 
 Back to \eqref{eq:Ext1} we deduce
  \beaa
&& \int_{\Sext_t} \big| \nabb^2 \phi\big|^2  -\frac{2}{r^2} |\nabb\phi|^2+ \frac{2}{t} \int_{r=t} | \nabb\phi|^2  \leq t^{-2} I+  \int_{\Sext_t} \big|L[\phi]\big|^2\\
 &\leq &  t^{-2}  \int_{\Sext_t}  (r^2-t^2)\pr^2_r \phi \lapp \phi+ t^{-2}\int_{\Sext}|\nabb\phi|^2 + t^{-1} \int_{r=t} | \nabb\phi|^2  +  \int_{\Sext_t} \big|L[\phi]\big|^2
 \eeaa
 Hence
 \be\lab{2025Aug4eqn32}
 \begin{split}
  \int_{\Sext_t} \big| \nabb^2 \phi\big|^2+  \frac{1}{t} \int_{r=t} | \nabb\phi|^2 &\les   t^{-2}  \int_{\Sext_t}  (r^2-t^2)\pr^2_r \phi \lapp \phi+  \int_{\Sext_t} \big|L[\phi]\big|^2\\ 
  &\quad  {+ t^{-2} \int_{\Sext} |\nabb\phi|^2}.
  \end{split}
 \ee
 Using the inequality  $(r^2-t^2)\le 2 t(r-t)+  2 (r-t)^2 $  we write, for some $\la>0$,
 \beaa
  \Big| \int_{\Sext_t}  (r^2-t^2)\pr^2_r \phi \lapp \phi\Big|&\les &  \int_{\Sext_t}  2 t(r-t) \big| \pr^2_r \phi \lapp \phi\big|
  + \int_{\Sext_t} (r-t)^2 |\pr^2\phi|^2 \\
  &\les & \la^{-1}   \int_{\Sext_t}  t^2  \big| \lapp \phi\big|^2 +\la \int_{\Sext_t}(r-t)^2    \big| \pr_r^2 \phi \big|^2+
   \int_{\Sext_t} (r-t)^2 |\pr^2\phi|^2 
 \eeaa
 Therefore,
 \be\label{2025Aug7eqn1}
   t^{-2}  \int_{\Sext_t}  (r^2-t^2)\pr^2_r \phi \lapp \phi  \les \la^{-1}    \int_{\Sext_t}   \big| \lapp \phi\big|^2+\la  t^{-2} \int_{\Sext_t}(r-t)^2    \big| \pr^2  \phi \big|^2
 \ee 
 Thus, taking $\la$ large we deduce  from \eqref{eq:Boch-lapp}, \eqref{2025Aug4eqn32}, and the above estimate
\be\lab{eq:Ext3}
 \begin{split}
 \int_{\Sext_t} \big| \nabb^2 \phi\big|^2  +  \frac{1}{t} \int_{r=t} | \nabb\phi|^2&\les   t^{-2} \int_{\Sext_t}(r-t)^2    \big| \pr^2  \phi \big|^2+ \int_{\Sext_t} \big|L[\phi]\big|^2    +  t^{-2} \int_{\Sext} |\nabla\phi|^2.
 \end{split}
 \ee
 It remains to  estimate the term $  \int_{\Sext_t}(r - t )^2|\pr_r^2 \phi|^2 $. To do this we appeal to the following.
\begin{lemma}
\lab{Lemma:exterior-energy}
The following estimate holds true, with $u= t-r$ in the exterior region $\Sext$  where $r\ge t$.
\bea
\lab{eq:eq5'}
\int_{\Sext_t}  |u|^2  |\pr^2  \phi|^2 &\les & \int_{\Si_0}  r^2  |\pr^2\phi|^2   +\big(\int_0^t \|   u  \pr \square \phi\|_{L^2(\Sext_t)} \big)^2 
\eea
\end{lemma}
Using  the Lemma we then deduce  from \eqref{eq:Ext3} and \eqref{eq:wavepsi}
\be
\begin{split}
 &\int_{\Sext_t} \big| \nabb^2 \phi\big|^2  +  \frac{1}{t} \int_{r=t} | \nabb\phi|^2\\ 
  &\les  t^{-2} \int_{\Si_0}  r^2  |\pr^2\phi|^2 + t^{-2}  \big(\int_0^t \|   u  \pr \square \phi\|_{L^2(\Sext_t)} \big)^2+\int_{\Sext}\big|L[\phi]\big|^2   + t^{-2} \int_{\Sext} |\nabla\phi|^2\\ 
 &\les  t^{-2} \int_{\Si_0}  r^2  |\pr^2\phi|^2 + t^{-2}  \big(\int_0^t \|   u  \pr \square \phi\|_{L^2(\Sext_t)} \big)^2 +   \int_{\Sext_t} t^{-2} \big|\pr \Ga^{\le 1 } \phi \big|^2 + \int_{\Sext_t} |f|^2.
 \end{split}
\ee 
Finally, in view of  \eqref{2025Aug4eqn31}, \eqref{2025Aug7eqn1} and \eqref{eq:Ext3}, we further deduce that
\beaa
 \int_{\Sext_t}    \big(      \frac{r^2}{t^2}-1 \big)\big|\nabb\pr_r\phi|^2 &\les &t^{-2} \int_{\Si_0}  r^2  |\pr^2\phi|^2 + t^{-2}  \big(\int_0^t \|   u  \pr \square \phi\|_{L^2(\Sext_t)} \big)^2
 \\
 &&+  t^{-2}  \int_{\Sext_t}  \big|\pr \Ga^{\le 1 } \phi \big|^2 + \int_{\Sext_t} |f|^2.
 \eeaa
 as stated. This ends the proof of  Lemma \ref{Le:Bochner-ext}.
 \end{proof}
\begin{proof}[Proof of Lemma \ref{Lemma:exterior-energy}]
We start  with  the usual identity  $\pr^\b Q_{\a\b}=2\pr_\a \phi \square\phi$,    for the energy momentum  tensor  $Q_{\a\b}=\pr_\a\phi\pr_\b \phi -\frac 1 2 \m_{\a\b} \pr_\la\phi \pr^\la \phi $, which we write in the form
\beaa
\pr_t Q_{00}-\pr^ i Q_{0i}=-\pr_t \phi \square\phi.
\eeaa
We calculate, with $u=t-r$, 
\beaa
\pr_t\big(|u|^2 Q_{00} \big)-\pr^i \big(|u|^2 Q_{0i} \big)&=&u^2\big( \pr_t Q_{00}-\pr^ i Q_{0i})+ 2 u \pr_t u Q_{00}-2 u\pr_i u Q_{0i}\\
&=& 2 u\big( \pr_t u Q_{00} -\pr_i u Q_{0i} \big)- u^2 \pr_t \phi \square \phi\\
&=& 2 u \big(Q_{00} +\frac{x^i}{r} Q_{0i} \big) - u^2 \pr_t \phi \square \phi\\
&=&2u\Big(\frac 1 2\big(| \pr_t\phi|^2 +|\pr_r\phi|^2 +|\nabb\phi|^2\big) + \pr_t\phi\pr_r\phi\Big)- u^2 \pr_t \phi \square \phi\\\
&=& u \big( |(\p_t +\p_r )\phi|^2 +|\nabb\phi|^2 \big) - u^2 \pr_t \phi \square \phi\
\eeaa
Thus, by integration in the  causal  region where  $u<0$,
\beaa
\frac {d}{dt}\int_{\Si_t\cap \{u<0\}} |u|^2  Q_{00} &=&\int _{\Si_t\cap \{u<0\}}  u \big(|(\p_t +\p_r )\phi|^2 +|\nabb\phi|^2 \big) -\int_{\Si_t\cap \{u<0\}}  |u|^2\pr_t \phi \square\phi 
\eeaa
or,
\be
\begin{split}
&\int_{\Si_t\cap \{u<0\}} |u|^2  Q_{00} -\int_0^t \int _{\Si_t\cap \{u<0\}}  u \big(| (\p_t +\p_r )\phi|^2 +|\nabb\phi|^2 \big)\\ 
&=\int_{\Si_0\cap \{u<0\}} |u|^2  Q_{00} -\int_0^t \int _{\Si_0\cap \{u<0\}} u^2  \pr_t \phi      \square \phi. \\ 
\end{split}
\ee
In particular,
\beaa
 \int_{\Si_t\cap \{u<0\}} |u|^2  |\pr \phi|^2 \les  \int_{\Si_0\cap \{u<0\}} |u|^2 |\pr\phi|^2 + \int_0^t \int _{\Si_0\cap \{u<0\}} u^2  |\pr_t \phi  |\, |    \square \phi|
\eeaa
By Gronwall  we deduce,
\bea
\lab{eq:eq5}
\int_{\Si_t\cap \{u<0\}} |u|^2  |\pr_t\phi|^2 &\les & \int_{\Si_0\cap \{u<0\}}  |u|^2  |\pr \phi|^2   +\Big(\int_0^t \|   u   \square \phi\|_{L^2(\Si_t\cap\{u<0\})} \Big)^2 
\eea
In the same manner, by commuting  with $\pr$, we yield
\beaa
\int_{\Si_t\cap \{u<0\}} |u|^2  |\pr^2 \phi|^2 &\les & \int_{\Si_0\cap \{u<0\}}  u^2  |\pr^2\phi|^2   +\Big(\int_0^t \|   u  \pr \square \phi\|_{L^2(\Si_t)\cap\{u<0\})} \Big)^2.
\eeaa
Hence finishing the proof of \eqref{eq:eq5'}. 
\end{proof}

 As another consequence of the  Lemma \ref{Le:2nd-pointwise}, 
  we derive the following  degenerate \emph{weighted} B\"ochner identity
\begin{lemma}\label{WDBI}
Let $\omega (t,r)$   be a   nonnegative $C^2$ weight function.   Solutions to  the equation   $\square \phi=f $,  verify 
\be\label{2025July25eqn1}
\begin{split}
 &\int_{\Si_t} \omega(t,r)\big[\big| \nabb^2 \phi\big|^2+   \big(1-     \frac{r^2}{t^2} \big)\big|\nabb\pr_r\phi|^2 + \big(1-     \frac{r^2}{t^2} \big)^2\big| \pr^2_r\phi|^2  \big]\\ 
 &\lesssim  \int_{\Si_t} t^{-4} \omega(t,r) \big| \Ga^{\le 2 } \phi \big|^2+
    (t^{-2} +  r^{-2})\langle \frac{r}{t}\rangle^{2} \omega(t,r) \big|\p_r \Ga^{\le 1 } \phi \big|^2   \\ 
    &\quad +  \omega(t,r) |f|^2 +  \big(\widetilde{\omega}(t,r)\big)^2 \big( {\omega}(t,r)\big)^{-1} |  \phi|^2,\\
\end{split}
\ee
where $\widetilde{\omega}(t,r)$ is uniquely determined by the radial derivatives of ${\omega}(t,r)$, see \eqref{variantomega} for its precise definition. 
\end{lemma}
\begin{remark}
 
In \eqref{2025July25eqn1}, we isolate the most troublesome term, $\partial_r \Gamma^{\leq 1} \phi$, noting that the final term in \eqref{2025July25eqn1} is a good term, see  \eqref{2025July31eqn22}.  This term arises specifically from the $\partial_r S\phi$ component of    $F[\phi]$ (see \eqref{eq:waveps-F1}) and  the low order term $ \pr_r\phi$ in $L(\phi)$ (see \eqref{eq:2nd-pointwise}).  {Reapplying a degenerate weighted B\"ochner identity to $\partial_r \Gamma^{\leq 1} \phi$, using once more the commutation properties of $\square$ with $\Ga$, allows us to further improve  its decay properties near the light cone\footnote{This is done first by  integrating  by parts with respect to $r$ to bound $\partial_r \Gamma^{\leq 1} \phi$ by $\partial_r^2 \Gamma^{\leq 1} \phi$ and $\Gamma^{\leq 1} \phi$, and then use the estimate \eqref{2025July25eqn1} to control $\partial_r^2 \Gamma^{\leq 1} \phi$ (see  the   precise estimates in  \eqref{2025June16eqn32} and \eqref{2025July28eqn52})}.
}
 {The estimate \eqref{2025July25eqn1} reveals that the weight function is unaffected by the inhomogeneous term $f$. Critically, the weighted Bochner identity is not applied to $f$, nor can it be. The treatment of $f$ occurs later, during the weighted Sobolev embedding, as explained in Remark \ref{remark1}.}
\end{remark}
\begin{proof}
We integrate  \eqref{eq:2nd-pointwise} on $\Si_t$ and derive the identity
\be\label{2025July25eqn2}
\begin{split}
 &\int_{\Si_t} \omega(t,r)\big[\big| \nabb^2 \phi\big|^2+   \big(1-     \frac{r^2}{t^2} \big)\big|\nabb\pr_r\phi|^2 + \big(1-     \frac{r^2}{t^2} \big)^2\big| \pr^2_r\phi|^2  \big]\\ 
 &\lesssim \int_{\Si_t } r^{-2}\langle \frac{r}{t} \rangle^2 \omega(t,r)|\p_r\phi|^2 + \int_{\Si_t }  \omega(t,r)  \big|L[\phi]\big|^2 + \big|  \int_{\Si_t } \omega(t,r) Div[\phi]  \big|. 
 \end{split}
 \ee
   To estimate  the right hand side we  recall that $ L[\phi]=F[\phi]+ f$  with $F[\phi]= \frac{1}{t^2} \big( S^2\phi -2r\p_r S\phi- S \phi\big)$, see \eqref{eq:waveps-F1}. Thus
\be\label{2025July25eqn3}
\int_{\Si_t }  \omega(t,r)  \big|L[\phi]\big|^2\lesssim  t^{-4} \int_{\Si_t} \omega(t,r) \big| \Ga^{\le 2 } \phi \big|^2+
  \int_{\Si_t}   t^{-2}\langle \frac{r}{t} \rangle^2 \omega(t,r) \big|\p_r \Ga^{\le 1 } \phi \big|^2 + \int_{\Si_t}\omega(t,r) |f|^2.
   \ee
  Recall \eqref{2025Aug4eqn11}. For the divergence part in \eqref{2025July25eqn2}, we do integration by parts.  As a result, we have
\be\label{2025July25eqn7}
\begin{split}
 \int_{\Si_t } \omega(t,r) Div[\phi] &=  -2\int_{\Si_t } \p_r \omega(t,r) \Big( \lapp \phi    \big(1-     \frac{r^2}{t^2} \big) \pr_r  \phi- r^{-1}|\nabb\phi|^2 \Big) \\ 
 \end{split}
 \ee
 For the first term on the right hand side, we first do integration by parts for the angular derivative and then use the  commutation rule $[\slashed \nabla, \p_r]= r^{-1}\slashed \nabla$. As a result, we have 
 \be\label{2025July25eqn11}
\begin{split}
-\int_{\Si_t } \p_r \omega(t,r)  \lapp \phi    \big(1-     \frac{r^2}{t^2} \big)  \pr_r  \phi& =  \int_{\Si_t } \p_r \omega(t,r)      \big(1-     \frac{r^2}{t^2} \big) \slashed\nabla \phi\cdot \slashed\nabla \pr_r  \phi\\
& =  \int_{\Si_t } \p_r \omega(t,r)      \big(1-     \frac{r^2}{t^2} \big)\big[ \frac{1}{2}\p_r\big(|\slashed\nabla \phi|^2\big) + r^{-1}|\slashed \nabla \phi|^2 \big].\\ 
 \end{split}
 \ee
 For the first term on the right hand side of \eqref{2025July25eqn11}, we do integration by parts for the radial derivative. As a result, we have
  \be\label{2025July25eqn13}
\begin{split}
& \int_{\Si_t } \p_r \omega(t,r)      \big(1-     \frac{r^2}{t^2} \big)  \p_r\big(|\slashed\nabla \phi|^2\big) = - \int_{\Si_t } r^{-2}\p_r \big(r^2\p_r \omega(t,r)      \big(1-     \frac{r^2}{t^2} \big)\big)   |\slashed\nabla \phi|^2. \\  
 \end{split}
 \ee
 To sum up, after combining the obtained estimates and equalities  in  \eqref{2025July25eqn2}--\eqref{2025July25eqn13}, we have 
 \be\label{2025July25eqn16}
\begin{split}
 & \int_{\Si_t} \omega(t,r)\big[\big| \nabb^2 \phi\big|^2+   \big(1-     \frac{r^2}{t^2} \big)\big|\nabb\pr_r\phi|^2 + \big(1-     \frac{r^2}{t^2} \big)^2\big| \pr^2_r\phi|^2  \big] \\ 
 &\lesssim  \int_{\Si_t}   t^{-4}\omega(t,r) \big| \Ga^{\le 2 } \phi \big|^2  +
   (t^{-2} +  r^{-2})\langle \frac{r}{t}\rangle^{2} \omega(t,r) \big|\p_r \Ga^{\le 1 } \phi \big|^2  \\ 
   &\quad  +  \omega(t,r) |f|^2  + \big| \int_{\Si_t} \widetilde{\omega}(t,r) |\slashed \nabla \phi|^2 \big| ,
 \end{split}
 \ee
 where  $\widetilde{\omega}(t,r)$ is defined as follows, 
 \be\label{variantomega}
   \widetilde{\omega}(t,r):=\frac{2}{r} \p_r \omega(t,r) -\frac{1}{2r^2} \p_r \big(r^2\p_r \omega(t,r)      \big(1-     \frac{r^2}{t^2} \big)\big). 
 \ee
 Moreover, for the last term in \eqref{2025July25eqn16}, we exploit the benefit of the first order angular derivative  by   doing integration by parts for the angular derivative first and then using the Cauchy-Schwarz inequality. As a result, we have
\be\label{2025July25eqn19}
\begin{split}
  \big| \int_{\Si_t} \widetilde{\omega}(t,r) |\slashed \nabla \phi|^2 \big| & = \big| \int_{\Si_t} \widetilde{\omega}(t,r)  \phi \slashed \Delta \phi  \big| \\ 
&\leq \big| \int_{\Si_t} \big(\widetilde{\omega}(t,r)\big)^2 \big( {\omega}(t,r)\big)^{-1}  |\phi|^2  \big|^{1/2} \big| \int_{\Si_t}  {\omega}(t,r)  |\slashed \Delta \phi|^2  \big|^{1/2}\\ 
&\leq c \int_{\Si_t}  {\omega}(t,r)  |\slashed \nabla^2 \phi|^2  + c^{-1}\int_{\Si_t} \big(\widetilde{\omega}(t,r)\big)^2 \big( {\omega}(t,r)\big)^{-1}  |\phi|^2,
 \end{split}
 \ee
where $c$ is some sufficiently small absolute constant.  

Combining the obtained estimate \eqref{2025July25eqn16} and \eqref{2025July25eqn19},  our desired estimate \eqref{2025July25eqn1} holds. It's worth to point out that, due to the estimate \eqref{2025July28eqn2}, there is no singularity for the weight function $\big(\widetilde{\omega}(t,r)\big)^2 \big( {\omega}(t,r)\big)^{-1} $ in our application. 
\end{proof}
\begin{remark}

It is important to note that the integration by parts in \eqref{2025July25eqn19} is crucial. In application, the weight function we use for $\omega(t,r)$ has the form $r^{\alpha} \chi_{\leq -10}((t-r)/(t-r_0))$, where $r_0$ is a fixed positive constant s.t., $r_0\gtrsim t$.  
 For simplicity in explanation, we assume $\alpha=0$. As a result of direct computation, we have
\[
(\textbf{Example case only}):  \big| \int_{\Si_t} \widetilde{\omega}(t,r) \big|\slashed \nabla \phi|^2 \big|\lesssim |t|^{-1}|t-r_0|^{-1}  \int_{\Si_t}  \big|\slashed \nabla \phi|^2. 
\]
 Moreover,  
\be\label{2025July31eqn22}
(\textbf{Example case only}):  \big| \int_{\Si_t} \big(\widetilde{\omega}(t,r)\big)^2 \big( {\omega}(t,r)\big)^{-1}  |\phi|^2  \big|\lesssim |t|^{-2}|t-r_0|^{-2}  \int_{\Si_t}  \big|  \phi|^2. 
\ee
The benefit of using integration by parts  is now evident in the above improved estimate.
\end{remark}

  \subsection{ Weighted Sobolev inequalities.}
 To prove Theorem \ref{Thm:main1}   we need  a weighted  version of the Sobolev Lemma  that  makes maximal use of the $L^2$ estimate \eqref{2025July25eqn1}.

 \begin{lemma}\label{Sobolevembedding}
Let $0<\delta\ll  1$. For any $x_0\in \R^3$, s.t., $|x_0|\geq 2^{-5} |t|$, we have
\be\label{2025June9eqn1}
\begin{split}
|\phi(t,x_0)|&\lesssim \delta^{-1}|x_0|^{-1}\big( \min\{|t|,\langle t-|x_0|\rangle\}\big)^{-1/2+2\delta}\|\p_r^{\leq 2} \phi(t,x)\|_{L^2(\R^3)}^{\delta} \\ 
&\quad\times  \big(\|\phi(t,x)\|_{L^2(\R^3)}  +  \|   \omega(t,|x|)    ||x|^2 \slashed\nabla^2 \phi(t,x)  |    \|_{L^2(\R^3)}\big)^{(1+\delta)/2} \\ 
    &\quad \times \big(\|\phi(t,x)\|_{L^2(\R^3)} +   \|   \omega(t,|x|)    \langle t-|x|\rangle  \p_r \phi(t,x)   \|_{L^2(\R^3)}\big)^{(1-3\delta)/2},
   \end{split}
  \ee
  where
  \be\label{weightfunctionsobo}
    \omega(t,|x|):=\chi_{\leq -10}\big(\frac{|x|-|x_0|}{\min\{\langle t-|x_0| \rangle, |t|\} } \big).
  \ee
 \end{lemma}
 \begin{remark}\label{remark1}
In the weighted Sobolev embedding \eqref{2025June9eqn1}, we carefully separate the angular and radial derivatives. This arrangement is crucial for applying the degenerate Bochner identity to both the second angular derivatives and the first radial derivative\footnote{By integrating by parts in $r$ once, we control the $L^2$-norm of the  first radial derivative part by the second radial derivatives and the zero derivative. We then apply the degenerate Bochner identity again for the second radial derivatives part.}. This allows us to obtain the crucial $\langle t\rangle^{1/2+}$ factor associated with the inhomogeneous term in the estimate \eqref{finalestphysicspace}.  
\end{remark}
 \begin{remark}
We put the restriction  $|x_0|\geq 2^{-5} |t|$ because the weighted Sobolev embedding will only be used for this case and the absolute constant $2^{-5}$ doesn't play a role as long as it's strictly less than $1$. For the case  $|x_0|\leq 2^{-5} |t|$, we use the uniform ellipticity of the operator $\Delta-r^2 t^{-2}\p_r^2.$
 \end{remark}

\begin{proof}

We proceed in two steps. Firstly, we prove a general  $L^\infty\rightarrow L^2$ type Sobolev embedding on the product space $\R\times \mathbb{S}^2$, which allows us to go from $L^\infty(\R\times \mathbb{S}^2)$-type space to $L^2(\R\times \mathbb{S}^2)$-type space. Lastly, we use the spherical coordinates  and  incorporate the weight function  to   go from $L^2(\R\times \mathbb{S}^2)$-type space to $L^2(\R^3)$-type space, which give us the desired weighted Sobolev embedding.

\textbf{Step 1.}\qquad A $L^\infty\rightarrow L^2$ type Sobolev embedding on the product space $\R\times \mathbb{S}^2. $

This step is devoted to prove the following estimate
\be\label{2025July24eqn6}
\begin{split}
 \|  f(x,w)\|_{L^\infty(\R\times  \mathbb{S}^2)}    &\lesssim  \delta^{-1}  \big( \|   f(x,w)\|_{L^2(\R\times  \mathbb{S}^2)} + \| \slashed\nabla^2  f(x,w)\|_{L^2(\R\times  \mathbb{S}^2)}  \big)^{(1+\delta)/2} \\ 
  &\quad \times    \| \p_{x}^{\leq 1}f(x,w)\|_{L^2(\R\times  \mathbb{S}^2)}^{(1-3\delta)/2} \| \p_{x}^{\leq 2}f(x,w)\|_{L^2(\R\times  \mathbb{S}^2)}^{ \delta}.
\end{split}
\ee

Before proving the above estimate, we introduce some necessary notation, which will be only used in this Lemma.

 Let $\{\lambda_j^2\}_{j=0}^{+\infty}$ be the set of eigenvalues of $-\Delta_{\mathbb{S}^2}$ and $\pi_j$ be the projection operator associated with the eigenvalue $\lambda_j^2$, i.e., $-\Delta_{\mathbb{S}^2} \pi_j h = \lambda_j^2 \pi_j h $, where $h: \mathbb{S}^2\rightarrow \R$. 
  We define the inhomogeneous  dyadic decomposition operator $ {P}^{sph}_n$, $n\in {\mathbb{Z}_+}$, with respect to the spectrum of $-\Delta_{\mathbb{S}^2}$ as follows, 
\[
  {P}^{sph}_0 f := \pi_0 f, \quad \forall n\geq 1,\quad    {P}^{sph}_n h :=\sum_{\lambda_j^2 \in [2^{n-1}, 2^{n})} \pi_j h.
\]
Moreover, we use ${P}^{rad}_k$ to denote the  inhomogeneous dyadic decomposition operator in $\R$ as follows, 
\[
 g:\R\rightarrow \R, \quad   {P}^{rad}_k g(x):=\left\{\begin{array}{ll}
   \displaystyle{ \int_{\R} e^{ix \cdot\xi} \chi_k(\xi) \widehat{g}(\xi) d \xi}  &  k\geq 1,k\in \Z_+\\ 
          \displaystyle{ \int_{\R} e^{ix \cdot\xi} \chi_{\leq 0}(\xi) \widehat{g}(\xi) d \xi}  & k=0. \\ 
\end{array}
   \right.
\]

For any $f(x,\omega)\in H^2(\R\times \mathbb{S}^2)$, we use $\mathcal{P}_{k,N}$ to denote the localization operator that localizes the frequency of $x\in \R$ by using the operator $   {P}^{rad}_k $ and localizes the eigenvalues of $-\Delta_{\mathbb{S}^2}$ associated with $\omega\in \mathbb{S}^2$ by using the operator ${P}_n^{sph}$, i.e., 
\[
   \mathcal{P}_{k,n}(f)(x,w):=P^{rad}_k  {P}_n^{sph} f. 
\]
 Note that,  from the classic  $L^\infty\rightarrow L^2$ type Sobolev embedding on $\R\times \mathbb{S}^2$, we have 
 \[
  \|  \mathcal{P}_{k,n}(f)(x,w) \|_{L^\infty(\R\times \mathbb{S}^2)}\lesssim 2^{k/2+n}   \| \mathcal{P}_{k,n}(f) \|_{L^2(\R\times \mathbb{S}^2)}.
 \] 
 From the above estimate, after using the partition of unity, we have
\be\label{2025June12eqn1}
\begin{split}
 \| f(x,w) \|_{L^\infty(\R\times  \mathbb{S}^2)}  & \lesssim \sum_{k, n\in \Z_+}     \|\mathcal{P}_{k,n}(f)  \|_{ L^\infty (\R\times  \mathbb{S}^2) }   \lesssim   \sum_{k, n\in \Z_+}    2^{k/2+n}     \| \mathcal{P}_{k,n}(f) \|_{L^2(\R\times  \mathbb{S}^2)}. \\ 
  \end{split}
\ee 
Let $M\in \Z_+$ be arbitrary. Note that, after expanding the product once and using the fact that $ \forall a,b\in\R, |ab|\leq a^2+b^2,$  we have 
\be\label{2025July24eqn1}
\begin{split}
&\big(\sum_{n_1, n_2=0}^M  2^{n_1} 2^{n_2/2} a_{n_1, n_2}\big)^2 =   \sum_{n_1, n_2=0}^M\sum_{m_1, m_2=0}^M   2^{n_1} 2^{n_2/2} a_{n_1, n_2}   2^{m_1} 2^{m_2/2} a_{m_1,m_2}  \\
&= \sum_{n_1, n_2=0}^M \sum_{m_1, m_2=0}^M  2^{ (1+\delta)(n_1+n_2/2) } a_{n_1,n_2}  2^{- \delta (m_1+m_2/2) } 2^{ (1+\delta)(m_1+m_2/2) } a_{m_1,m_2}  2^{- \delta (n_1+n_2/2) }  \\  
 & \lesssim     \sum_{n_1, n_2=0}^M \sum_{m_1, m_2=0}^M  2^{ (1+\delta)(2n_1+n_2) } a_{n_1,n_2}^2 2^{- \delta (2m_1+m_2) } +  2^{ (1+\delta)(2m_1+m_2) } a_{m_1,m_2}^2 2^{- \delta (2n_1+n_2) }   \\
\end{split}
\ee 
Since the geometric series is summable, after  using the Taylor expansion  of $2^{-\delta}$ and the   $ (2/(1+\delta), 2/(1-\delta))$ type H\"older inequality , we have  
\be\label{2025July24eqn2}
\begin{split}
\eqref{2025July24eqn1}&   \lesssim   \big(\frac{1}{1-2^{- \delta}}\big)^2  \big( \sum_{n_1, n_2=0}^M  2^{ (1+\delta)(2n_1+n_2)} a_{n_1,n_2}^2\big) \\ 
&\lesssim \delta^{-2}  \big( \sum_{n_1, n_2=0}^M  ( 2^{2n_1} a_{n_1,n_2} )^{(1+\delta)}  ( 2^{(1+\delta)n_2/(1-\delta)} a_{n_1,n_2} )^{(1-\delta)}     \big) \\ 
& \lesssim  \delta^{-2} \big(\sum_{n_1, n_2=0}^M 2^{4n_1 } a_{n_1, n_2}^2 \big)^{(1+\delta)/2}\big( \sum_{n_1, n_2=0}^M 2^{2n_2(1+\delta)/(1-\delta)} a_{n_1, n_2}^2 \big)^{(1-\delta)/2}.\\  
\end{split}
\ee
For the second term in the last line of \eqref{2025July24eqn2}, we use the  $((1-\delta)/(1-3\delta),(1-\delta)/(2\delta))$ type H\"older inequality. As a result, we have  
\be\label{2025July24eqn4}
\begin{split}
 &\sum_{n_1, n_2=0}^M 2^{2n_2(1+\delta)/(1-\delta)} a_{n_1, n_2}^2 \\ 
 &=  \sum_{n_1, n_2=0}^M \big(2^{4n_2} a^2_{n_1, n_2} \big)^{2\delta/(1-\delta)}  \big(2^{2n_2} a_{n_1, n_2}^2 \big)^{(1-3\delta)/(1-\delta)} \\ 
 &\lesssim  \big( \sum_{n_1, n_2=0}^M 2^{4n_2} a^2_{n_1, n_2} \big)^{2\delta/(1-\delta)} \big( \sum_{n_1, n_2=0}^M  2^{2n_2} a_{n_1, n_2}^2 \big)^{(1-3\delta)/(1-\delta)}.
\end{split}
\ee
To sum up, after combining the obtained estimates \eqref{2025July24eqn2} and \eqref{2025July24eqn4}, we have 
\be
\begin{split}
 \sum_{n_1, n_2=0}^M  2^{n_1} 2^{n_2/2} a_{n_1, n_2}  & \lesssim  \delta^{-1} \big(\sum_{n_1, n_2=0}^M 2^{4n_1 } a_{n_1, n_2}^2 \big)^{(1+\delta)/4}\big( \sum_{n_1, n_2=0}^M 2^{2n_2} a_{n_1, n_2}^2 \big)^{(1-3\delta)/4}\\ 
&\quad \times \big( \sum_{n_1, n_2=0}^M 2^{ 4n_2} a_{n_1, n_2}^2 \big)^{\delta/2}.\\  
\end{split}
\ee

  From the above estimate, by letting $M\rightarrow +\infty$, we conclude 
\be\label{2025June15eqn31}
\begin{split}
   \sum_{n_1, n_2\in \Z_+}   2^{n_1} 2^{n_2/2} a_{n_1, n_2}   &\lesssim  \delta^{-1}  \big( \sum_{n_1, n_2\in \Z_+}  2^{4n_1  } a_{n_1, n_2}^2 \big)^{(1+\delta)/4}\big( \sum_{n_1, n_2\in \Z_+} 2^{2n_2}  a_{n_1, n_2}^2 \big)^{(1-3\delta)/4}\\ 
     &\quad \times \big( \sum_{n_1, n_2\in \Z_+} 2^{4n_2}  a_{n_1, n_2}^2 \big)^{\delta/2}. 
\end{split}
\ee
Therefore, from the above estimate and the obtained estimate \eqref{2025June12eqn1},  we have
\be\label{2025May27eqn14}
\begin{split}
   &\| f(x,w) \|_{L^\infty(\R\times  \mathbb{S}^2)} \lesssim  \delta^{-1}      (   \sum_{k, n\in \Z_+}    2^{4n }\| \mathcal{P}_{k,n}(f)   \|_{L^2(\R\times  \mathbb{S}^2)}^2)^{(1+\delta)/4} \\ 
  &\quad \times (   \sum_{k, n\in \Z_+}    2^{2k}\|   \mathcal{P}_{k,n}(f)  \|_{L^2(\R\times  \mathbb{S}^2)}^2 )^{(1-3\delta)/4} (   \sum_{k, n\in \Z_+}    2^{4k}\|  \mathcal{P}_{k,n}(f)  \|_{L^2(\R\times  \mathbb{S}^2)}^2 )^{ \delta/2}\\  
  &\lesssim  \delta^{-1}  \big( \|   f(x,w)\|_{L^2(\R\times  \mathbb{S}^2)} + \| \slashed\nabla^2  f(x,w)\|_{L^2(\R\times  \mathbb{S}^2)}  \big)^{(1+\delta)/2} \\ 
  &\quad \times    \| \p_{x}^{\leq 1}f(x,w)\|_{L^2(\R\times  \mathbb{S}^2)}^{(1-3\delta)/2} \| \p_{x}^{\leq 2}f(x,w)\|_{L^2(\R\times  \mathbb{S}^2)}^{ \delta}.
\end{split}
\ee
Hence finishing the proof of  our desired estimate \eqref{2025July24eqn6}.

\textbf{Step 2.}\qquad Weighted Sobolev inequality  on $\R^3$.

Denote $q_0:=|t|-|x_0|$, $q= t-r$. Based on the possible size of $q_0$, we proceed in two steps as follows.

\textbf{Step 2A}.\quad If $|q_0|\leq 1$.

Let 
\be\label{2025Aug9eqn11}
f(t,r,\omega):=\phi(t,r\omega) \chi_{\leq -20}\big(\frac{r-|x_0|}{\langle t-|x_0|\rangle } \big)  .
\ee
From the estimate \eqref{2025July24eqn6} obtained in 
{Step 1}  and the volume measure in $\R^3$ in spherical coordinates,   we have
\be\label{2025July24eqn14}
\begin{split}
 |\phi(t,x_0)| & =\big| f(t,|x_0|, x_0/|x_0|)\big| \\ 
 &\lesssim  \delta^{-1}  \big( \|   f(t,r,w)\|_{L^2(\R\times  \mathbb{S}^2)} + \| \slashed\nabla^2  f(t,r,w)\|_{L^2(\R\times  \mathbb{S}^2)}  \big)^{(1+\delta)/2} \\ 
  &\quad \times    \| \p_{r}^{\leq 1}f(t,r,w)\|_{L^2(\R\times  \mathbb{S}^2)}^{(1-3\delta)/2} \| \p_{r}^{\leq 2}f(t,r,w)\|_{L^2(\R\times  \mathbb{S}^2)}^{ \delta}. \\ 
   & \lesssim \delta^{-1} |x_0|^{-1} \big(    \int_{\R} \int_{\mathbb{S}^2}  \big(\omega(t,|x|) \big)^2  \big(|   \phi(t,r\omega)|^2  +  |r^2 \slashed\nabla^2 \phi(t,rw)  |^2 \big) r^2 dr d\omega \big)^{(1+\delta)/2}\\ 
     &\quad \times \big[  \int_{\R} \int_{\mathbb{S}^2}  \big(\omega(t,|x|) \big)^2   \big(|   \phi(t,r\omega)|^2  +   |   \p_r \phi(t,r\omega)|^2 \big)   r^2 dr d\omega  \big]^{(1-3\delta)/2}\\ 
    &\quad \times \big[  \int_{\R} \int_{\mathbb{S}^2}  \big(\omega(t,|x|) \big)^2     \big(|   \phi(t,r\omega)|^2 +   |   \p_r \phi(t,r\omega)|^2 +   |   \p_r^2 \phi(t,r\omega)|^2 \big)   r^2 dr d\omega  \big]^{\delta},  \\  
   \end{split}
\ee
where $\omega(t,|x|)$ is defined in \eqref{weightfunctionsobo}. 
In the above estimate, we used the fact that the angular derivatives of the cutoff function are zero and the radial derivatives of the cutoff function are much less than $1$. If $\p_r$ hits   the cutoff function $\chi_{\leq -20}\big(\frac{r-|x_0|}{\langle t-|x_0|\rangle } \big)$ of $f$, see \eqref{2025Aug9eqn11}, we simply yield a good lower order term.

\textbf{Step 2B}.\quad If $|q_0|\geq 1$.

Let 
\[
   v(t,a, \omega):= \chi_{\leq -20}(a) \phi (t, (t-  \min\{q_0, t\}(1+a) ) \omega ).
\]
   From the estimate \eqref{2025July24eqn6} obtained in 
Step 1  and the volume measure in $\R^3$ in spherical coordinates,   we have   
\be\label{2025July24eqn15}
\begin{split}
 & |\phi(t, x_0)|^2= |v(t, 0, \omega_0)|^2 \\ 
   &\lesssim \delta^{-2} \big[   \int_{\R} \int_{\mathbb{S}^2} |   v(t,a,\omega)|^2  + |\slashed \nabla^2  v(t,a,\omega) | d \omega  d a\big]^{(1+\delta)/2}  \\
   &\quad \times \big[  \int_{\R} \int_{\mathbb{S}^2} | \p_a^{\leq 1} v(t,a,\omega)|^2 d \omega  d a \big]^{(1-3\delta)/2} \big[   \int_{\R} \int_{\mathbb{S}^2} | \p_a^{\leq 2} v(t,a,\omega)|^2 d \omega  d a \big]^{\delta} \\ 
   &\lesssim \delta^{-2} |x_0|^{-2} ( \min\{q_0, t\})^{-1} \big[   \int_{\R} \int_{\mathbb{S}^2}  \big(\omega(t,|x|)\big)^2  \big(|   \phi(t,r\omega)|^2  +  |r^2 \slashed\nabla^2 \phi(t,rw)  |^2 \big) r^2 dr d\omega \big]^{(1+\delta)/2}\\ 
    &\quad \times \big[  \int_{\R} \int_{\mathbb{S}^2}   \big(\omega(t,|x|)\big)^2 \big(|   \phi(t,r\omega)|^2  +   | |t-r|  \p_r \phi(t,r\omega)|^2 \big)   r^2 dr d\omega  \big]^{(1-3\delta)/2}\\ 
    &\quad \times \big[  \int_{\R} \int_{\mathbb{S}^2}   \big(\omega(t,|x|)\big)^2 \big(|   \phi(t,r\omega)|^2 +   | |t-r|  \p_r \phi(t,r\omega)|^2 +   | |t-r|^2  \p_r^2 \phi(t,r\omega)|^2 \big)   r^2 dr d\omega  \big]^{\delta},\\ 
\end{split}
\ee
where $\omega(t,|x|)$ is defined in \eqref{weightfunctionsobo}.  To sum up, after combining the obtained estimates \eqref{2025July24eqn14} and \eqref{2025July24eqn15}, we have 
\be\label{2025July30eqn35}
\begin{split}
 &|x_0|( \min\{q_0, t\})^{1/2}|\phi(t,x_0)| \\ 
 &\lesssim \delta^{-1}  \| \omega(t,|x|)  \big(|   \phi(t,x)|   +  ||x|^2 \slashed\nabla^2 \phi(t,x)  | \big)  \|_{L^2(\R^3)}^{(1+\delta)/2}\\ 
    &\quad \times  \| \omega(t,|x|)  \big(|   \phi(t,x)|   +  \langle t-|x|\rangle | \p_r \phi(t,x)  | \big)  \|_{L^2(\R^3)}^{(1-3\delta)/2}\\ 
    &\quad \times   \|  \omega(t,|x|) \big(|   \phi(t,x)|   +  \langle t-|x|\rangle | \p_r \phi(t,x)  |+  \langle t-|x|\rangle^2 | \p_r^2 \phi(t,x)  | \big)  \|_{L^2(\R^3)}^{\delta} \\ 
    &\lesssim \delta^{-1}( \min\{q_0, t\})^{2\delta} \big(\|\phi(t,x)\|_{L^2(\R^3)}  +  \|  \omega(t,|x|)   ||x|^2 \slashed\nabla^2 \phi(t,x)  |    \|_{L^2(\R^3)}\big)^{(1+\delta)/2} \\ 
    &\quad \times \big(\|\phi(t,x)\|_{L^2(\R^3)} +   \| \omega(t,|x|) \langle t-|x|\rangle  \p_r \phi(t,x)   \|_{L^2(\R^3)}\big)^{(1-3\delta)/2}\|\p_r^{\leq 2} \phi(t,x)\|_{L^2(\R^3)}^{\delta}.
\end{split}
\ee
Hence finishing the proof of our desired weighted Sobolev embedding \eqref{2025June9eqn1}.
\end{proof}
 
  \subsection{Proof of \eqref{finalestphysicspace} in  Theorem \ref{Thm:main1}}
  
  Based on the size of $x_0$, we proceed in two steps as follows.

\textbf{Step 1.}\qquad If $x_0\in supp(\chi_{\leq -8}(\cdot/|t|)) $.

Let $u $ be $\phi$ or derivatives of $\phi$. Recall \eqref{2025July28eqn21}. We have
\be\label{2025July28eqn22}
 \Delta u  - \frac{x_ix_j\p_{i}\p_{j}}{t^2} u = \frac{  S^2 u -2r \p_r S u     - S u + 2r\p_r u }{t^2} + \square u.
\ee
 Note that, 
\be\label{2025June15eqn1}
\begin{split}
& \int_{\Si_t}  |\Delta u -t^{-2} x_i x_j\p_{i}\p_{j} u|^2  \chi_{\leq -5}(\frac{|x|}{t})   \\ 
&=  \int_{\Si_t} \big[\big(|\Delta u|^2 + t^{-4}(x_ix_j)^2 (\p_{i}\p_{j}u)^2  \big)  - 2 t^{-2}(x_ix_j)  (\p_{i}\p_{j}u)\Delta u \big] \chi_{\leq -5}(\frac{|x|}{t}).   \\ 
\end{split}
\ee 
From the Cauchy-Schwarz inequality, we have
\be\label{2025July28eqn23}
\begin{split}
\big|\int_{\Si_t} t^{-2}(x_ix_j)  (\p_{i}\p_{j}u)\Delta u \chi_{\leq -5}(\frac{|x|}{t})\big|&\leq 2^{-8}\int_{\Si_t} \big|\nabla^2 u\big|\big|\Delta u\big| \chi_{\leq -5}(\frac{|x|}{t}) \\ 
&\leq 2^{-8}\big[ \int_{\Si_t} \big|\nabla^2 u\big|^2 \chi_{\leq -5}(\frac{|x|}{t})+ \big|\Delta u\big|^2 \chi_{\leq -5}(\frac{|x|}{t})  \big].
\end{split}
\ee 
Moreover, for any $i,j\in\{1,2,3\}$, by doing integration by parts in $x$, we have
\be\label{2025July28eqn24}
  \big| \int_{\Si_t} \p_i \p_j u \p_i\p_j u  \chi_{\leq -5}(\frac{|x|}{t}) - \p_i \p_i u \p_j\p_j u  \chi_{\leq -5}(\frac{|x|}{t})\big| \lesssim |t|^{-2} \int_{\Si_t} | \nabla   u|^2 
\ee
After combining the obtained estimates \eqref{2025July28eqn22}--\eqref{2025July28eqn24}, we conclude
\be\label{2025July28eqn25}
\begin{split}
 \int_{\R^3} |\nabla^2 u(t,x)| \chi_{\leq -5}(\frac{|x|}{t}) d x &\lesssim  \int_{\Si_t}  |\Delta u -t^{-2} x_i x_j\p_{i}\p_{j} u|^2  \chi_{\leq -5}(\frac{|x|}{t}) + |t|^{-2} \int_{\Si_t} | \nabla   u|^2 \\  
 &\lesssim  |t|^{-2} \int_{\Si_t} | \pr \Gamma^{\leq 1}   u|^2 + \int_{\Si_t} |\square u|^2. 
\end{split}
\ee 
Note that, from the $L^\infty\rightarrow L^2$-type Sobolev embedding in $3D$, we have
\be\label{2025July28eqn61}
  \|u (t,x)\|_{L^\infty_x}  \lesssim     \| \nabla \Gamma^{\leq 1}u(t,x)\|_{L^2_x}. 
\ee
From the above Sobolev embedding   and the obtained estimate \eqref{2025July28eqn25}, we conclude  
\be\label{smallxcase}
\begin{split}
 |\nabla \phi(t,x_0)|^2     \chi_{\leq -8}(\frac{|x_0|}{t}) 
&\lesssim    \int_{\R^3 } |\nabla^2\Gamma^{\leq 1} \phi|^2 \chi_{\leq -5}(\frac{|x|}{t}) d x\\
&\lesssim     |t|^{-2} \int_{\Si_t} | \pr \Gamma^{\leq 2}   \phi |^2 + \int_{\Si_t} |\square  \Gamma^{\leq 1}   \phi|^2.
\end{split}
\ee 
Moreover, after using the fact that $\p_t u = t^{-1}(Su-r\p_r u)$ and the Sobolev embedding \eqref{2025July28eqn61}, from the above   estimate \eqref{smallxcase}, we have 
\be\label{smallxcasetder}
\begin{split}
 |\p_t\phi(t,x_0)|^2  \chi_{\leq -8}(\frac{|x_0|}{t}) & \lesssim t^{-2}|S\phi(t,x_0)|^2  \chi_{\leq -10}(\frac{|x_0|}{t})  +  |\nabla \phi(t,x_0)|^2      \chi_{\leq -10}(\frac{|x_0|}{t}) \\ 
 &\lesssim     |t|^{-2} \int_{\Si_t} | \pr \Gamma^{\leq 2}   \phi |^2 + \int_{\Si_t} |\square  \Gamma^{\leq 1}   \phi|^2.
\end{split}
\ee

\textbf{Step 2.}\qquad If    $x_0\in supp(\chi_{> -8}(\cdot/|t|)) $.

Denote $q_0:=|t|-|x_0|$ $q= t-r$. 
Based on the possible size of $q_0$, we proceed in two steps as follows.

\textbf{Step 2A}.\quad If $|q_0|\leq 1$.

From the degenerated weighted  B\"ochner equality  \eqref{2025July25eqn1} in Lemma \ref{WDBI} with weight     $\om=r^4 \chi^2_{\geq -20}(|x|/t)$,   we have 
\be\label{2025June14eqn52}
\begin{split}
&\| |r|^{2}   \slashed\nabla_{}^{2} \phi  (x) \chi_{\geq -20}(|x|/t)\|_{L^2(\R^3 )}^2  \lesssim t^2   \int_{\Si_t}  | \Gamma^{\leq 2 }  \phi |^2 d x + t^4  \int_{\Si_t}  |\square  \phi|^2 d x . \\ 
   \end{split}
\ee
After combining the above estimate  \eqref{2025June14eqn52} and the weighted Sobolev embedding \eqref{2025June9eqn1} in Lemma \ref{Sobolevembedding}, we conclude 
\be\label{2025June14eqn53}
\begin{split}
  |\phi(t,x_0)|&\lesssim \delta^{-1} |t|^{-1}  \|\p_r^{\leq 2} \phi(t,x)\|_{L^2(\R^3)}^{ (1-\delta)/2}\big(    t  \|\Gamma^{\leq 2}\phi\|_{L^2} + t^2 \| \square \phi\|_{L^2}  \big)^{(1+\delta)/2} \\ 
  &\lesssim \delta^{-1} \big[|t|^{-(1-\delta)/2} \|\Gamma^{\leq 2}\phi\|_{L^2} + |t|^{\delta} \|\Gamma^{\leq 2}\phi\|_{L^2}^{ (1-\delta)/2} \| \square \phi\|_{L^2}^{(1+\delta)/2}\big]\\ 
  &\lesssim \delta^{-1} \big[  |t|^{-(1-\delta)/2} \|\Gamma^{\leq 2}\phi\|_{L^2}  + |t|^{ (1+\delta)/2 } \| \square \phi\|_{L^2}\big]. 
\end{split}
\ee

\textbf{Step 2B}.\quad If $|q_0|\geq 1$.

From the degenerated weighted  B\"ochner equality  \eqref{2025July25eqn1} in Lemma \ref{WDBI} with weight
$\om=\chi_{\leq -10}\big(\frac{r-|x_0|}{\langle t-|x_0|\rangle } \big) r^{4}$,   we have
\be\label{2025June9eqn11}
\begin{split}
& \int_{\Si_t}  \chi_{\leq -10}\big(\frac{r-|x_0|}{\langle t-|x_0|\rangle } \big) r^{4}  |    \slashed \nabla^{2}    \phi  (t, x)|^2 dx \\ 
&\lesssim |q_0|^{-2} |t|^2 \int_{\Si_t}    |\Gamma^{\leq 2} \phi |^2 d x + |t|^4  \int_{\Si_t}   |\square \phi |^2 d x + \int_{\Si_t}  \chi_{\leq -10}\big(\frac{r-|x_0|}{ \langle t-|x_0|\rangle } \big)  r^2 |\p_r (S\phi)|^2 d x.  \\  
\end{split}
\ee

To estimate the last term in the above estimate, we first do integration by parts in $r$ once and then use the Cauchy-Schwartz inequality. As a result, we have
\be\label{2025June16eqn32}
\begin{split}
& \int_{\Si_t}   \chi_{\leq -10}\big(\frac{r-|x_0|}{ \langle t-|x_0|\rangle } \big)  r^2 |\p_r (S\phi)|^2 d x\\ 
&\lesssim \big( \int_{\Si_t}  \chi_{\leq -10}\big(\frac{r-|x_0|}{ \langle t-|x_0|\rangle } \big)  r^4 |\p_r^2 (S\phi )|^2   \big)^{1/2} \| S\phi\|_{L^2}  + |q_0|^{-1}|t|^2   \| \Gamma^{\leq 2}\phi\|_{L^2}^2  . \\
 \end{split}
\ee
Let 
\be
\begin{split}
  {\omega}(t,r) & :=  \chi_{\leq -10}\big(\frac{r-|x_0|}{t-|x_0|} \big) |t-r|^2 r^2  \frac{1}{\big(1-\frac{r^2}{t^2}\big)^2 } =\chi_{\leq -10}\big(\frac{r-|x_0|}{t-|x_0|} \big) \big( \frac{t^2r}{t+r}\big)^2. \\ 
  \end{split}
\ee
From the degenerated weighted  B\"ochner equality  \eqref{2025July25eqn1} in Lemma \ref{WDBI},   we have
\be\label{2025July28eqn52}
\begin{split}
 \int_{\Si_t}  \chi_{\leq -10}\big(\frac{r-|x_0|}{ \langle t-|x_0|\rangle }  r^4 |\p_r^2 (S\phi )|^2 
 &\lesssim |q_0|^{-2} |t|^2\big[ \int_{\Si_t} \big(1-\frac{r^2}{t^2}\big)^2  \omega(t,r)   |\p_r^2 (S\phi )|^2\big]\\ 
 &\lesssim |q_0|^{-2} |t|^2\big[ t^2 \|\Gamma^{\leq 3} \phi\|_{L^2}^2 + t^4\|\square S\phi\|_{L^2}^2  \big].
\end{split}
\ee
After combining the above estimate and the obtained estimate \eqref{2025June16eqn32}, we have 
\be\label{2025July28eqn54}
\begin{split}
  \int_{\Si_t}   \chi_{\leq -10}\big(\frac{r-|x_0|}{ \langle t-|x_0|\rangle } \big)  r^2 |\p_r (S\phi)|^2 d x  
&\lesssim |q_0|^{-1}|t|^2   \| \Gamma^{\leq 3}\phi\|_{L^2}^2 + |q_0|^{-1}|t|^3  \| S\phi\|_{L^2} \|\square S\phi\|_{L^2}\\ 
&\lesssim |q_0|^{-1}|t|^2   \| \Gamma^{\leq 3}\phi\|_{L^2}^2 + |q_0|^{-1}|t|^4 \|\square S\phi\|_{L^2}^2. 
\end{split}
\ee
Thanks to the cutoff function $\chi_{\leq -10}\big(\frac{|x|-|x_0|}{\langle t-|x_0| \rangle} \big)$, we have $t-r\sim q_0$ and $r\sim t$. As a result, similar to  the above obtained estimate  \eqref{2025July28eqn54}, we have 
\be\label{2025July28eqn55}
\begin{split}
  & \| \chi_{\leq -10}\big(\frac{|x|-|x_0|}{\langle t-|x_0| \rangle} \big)    \langle t-|x|\rangle  \p_r \phi(t,x)   \|_{L^2(\R^3)} \\ 
  & \lesssim |q_0| |t|^{-1} \| \chi_{\leq -10}\big(\frac{|x|-|x_0|}{\langle t-|x_0| \rangle} \big)   r  \p_r \phi(t,x)   \|_{L^2(\R^3)} \\ 
  &\lesssim |q_0| |t|^{-1}\big[|q_0|^{-1/2}|t|   \| \Gamma^{\leq 2}\phi\|_{L^2}  + |q_0|^{-1/2}|t|^2 \|\square  \phi\|_{L^2}  \big] \\ 
  & \lesssim  |q_0|^{1/2}    \| \Gamma^{\leq 2}\phi\|_{L^2}  + |q_0|^{1/2}|t|^{ } \|\square  \phi\|_{L^2}.
\end{split}
\ee
After combining the above estimates \eqref{2025June9eqn11}, \eqref{2025July28eqn54} and \eqref{2025July28eqn55} with the weighted Sobolev embedding  \eqref{2025June9eqn1} in Lemma \ref{Sobolevembedding}, we conclude 
\be\label{2025July28eqn56}
\begin{split}
|\phi(t,x_0)|&\lesssim \delta^{-1}|t|^{-1} |q_0|^{-1/2+2\delta}\|\p_r^{\leq 2} \phi(t,x)\|_{L^2(\R^3)}^{\delta}   \big(|q_0|^{1/2}    \| \Gamma^{\leq 2}\phi\|_{L^2}  + |q_0|^{1/2}|t|^{ } \|\square  \phi\|_{L^2}\big)^{(1-3\delta)/2} \\ 
&\quad \times \big( |q_0|^{-1/2}|t| \| \Gamma^{\leq 3}\phi\|_{L^2} + |t|^2    \|\square \phi\|_{L^2} + |q_0|^{-1/2} |t|^{2}  \|\square S\phi\|_{L^2}  \big)^{(1+\delta)/2}\\ 
&\lesssim \delta^{-1} \big[|t|^{-(1-\delta)/2} |q_0|^{-1/2+\delta}\|\Gamma^{\leq 3} \phi(t,x)\|_{L^2 }  +  |t|^{(1+\delta)/2}\|\square \Gamma^{\leq 1} \phi\|_{L^2}\big] . 
\end{split}
\ee
To sum up,     our desired estimate \eqref{finalestphysicspace} holds from \eqref{smallxcase},\eqref{smallxcasetder}, \eqref{2025June14eqn53}, and \eqref{2025July28eqn56}. Hence finishing the proof of Theorem \ref{Thm:main1}. 

  \subsection{Proof of \eqref{finalestphysicspaceexterior} in  Theorem \ref{Thm:main1} }

  Based on the size of $x_0$, we proceed in two steps as follows. 

\vo 

\textbf{Step 1.}\qquad If $x_0\in supp(\chi_{\leq -8}(\cdot/|t|)) $. 

\vo

This case follows exactly same as what we did in previous subsection. The obtained estimates \eqref{smallxcase} and  \eqref{smallxcasetder} are still valid. Therefore, we have
\be
 |\pr \phi(t,x_0)|^2  \chi_{\leq -8}(\frac{|x_0|}{t})  \lesssim     |t|^{-2} \int_{\Si_t} | \pr \Gamma^{\leq 2}   \phi |^2 + \int_{\Si_t} |\square  \Gamma^{\leq 1}   \phi|^2.
\ee

\vo 

\textbf{Step 2.}\qquad If $x_0\in supp(\chi_{> -8}(\cdot/|t|)) $. 

\vo

Choose $\omega(t,|x|) $ to be the one prescribed in \eqref{weightfunctionsobo}. Since $0\leq \omega(t,r)\leq 1$, see \eqref{2025July28eqn1}, we have 
 {\[
\begin{split}
 \int_{\Si_t} \omega(t,r)  \big| \nabb^2 \phi\big|^2  
&\lesssim \underbrace{\int_{\Si_t} \omega(t,r)\big[\big| \nabb^2 \phi\big|^2+   \big(1-     \frac{r^2}{t^2} \big)\big|\nabb\pr_r\phi|^2    \big]}_{\text{controlled by Lemma \ref{WDBI}}}  + \underbrace{\int_{\Sext_t} \big(     \frac{r^2}{t^2}-1 \big)\big|\nabb\pr_r\phi|^2}_{\text{controlled by Lemma \ref{exteriorestimate}}}. 
\end{split}
\]}
From the weighted Bochner identity \eqref{2025July25eqn1} in  Lemma \ref{WDBI}, we have
\[
\int_{\Si_t} \big(\omega(t,|x|)\big)^2  \big[\big| \nabb^2 \phi\big|^2+   \big(1-     \frac{r^2}{t^2} \big)\big|\nabb\pr_r\phi|^2\big] \lesssim |t|^{-2} \|\Gamma^{\leq 2}\phi\|_{L^2}^2 +\| f\|_{L^2}^2.
\]
After combining the above estimate with the   estimate \eqref{2025Aug4eqn81} in the exterior region, we have
\be\label{2025Aug4eqn83}
\begin{split}
  \int_{\Si_t}  \omega(t,|x|) \big| \nabb^2 \phi\big|^2&\lesssim |t|^{-2} \big( \|\Gamma^{\leq 2}\phi(t)\|_{L^2}^2+  \|r^2\Gamma^{\leq 2}\phi(0)\|_{L^2}^2 \big)\\ 
  &\quad  +\| f\|_{L^2}^2  + t^{-2}  \big(\int_0^t \|   u  \pr  f\|_{L^2(\Sext_t)} \big)^2.
\end{split}
\ee
From the above estimate and the weighted Sobolev embedding \eqref{2025June9eqn1} in  Lemma \ref{Sobolevembedding}, we have 
 \be
 \begin{split}
 |\pr \phi(t,x_0)|  \chi_{> -8}(\frac{|x_0|}{t}) & \lesssim \delta^{-1} |t|^{-1/2} |x_0|^{\delta} \|\Gamma^{\leq 3}\phi(t) \|_{L^2}^{(1-\delta)/2}\big[    \big( \|\Gamma^{\leq 3}\phi(t)\|_{L^2} +  \|r^2\Gamma^{\leq 3}\phi(0)\|_{L^2}\big) \\ 
 &\quad   + |t|\| \Gamma^{\leq 1} f\|_{L^2} +    \big(\int_0^t \|   u  \pr  \Gamma^{\leq 1} f\|_{L^2(\Sext_t)} \big)\big]^{(1+\delta) /2}.
 \end{split}
\ee
Hence finishing the proof of our desired estimate \eqref{finalestphysicspaceexterior}.
  \section{Proof of Theorem \ref{Thm:main2}}

As discussed in Remark \ref{geometricmeasure}, the proof of Theorem \ref{Thm:main2} hinges on a geometric observation concerning the localized set $S_{k,l}(t,x)$. For the reader's convenience, we restate its definition below:   \be\label{mainestvol}
\begin{split}
  \forall x \in \R^3, t\in \R/\{0\},  \mu\in\{+,-\},  \qquad & S_{k,l}(t,x):=\big\{\xi: | 1+ \frac{\mu x\cdot\xi}{t|\xi|} |\leq  2^{l}, |\xi|\in  [2^{k-1}, 2^{k}] \big\}. 
\end{split}
\ee

First, we explain the relevance of this set. From the standard Duhamel's formula for the wave equation, it suffices to estimate the associated half-wave. Specifically, we define
\be\label{aug14eqn32}
U(t):= (\p_t - i |\nabla|) \phi, \quad  V(t):= e^{i t |\nabla|} U(t). 
\ee
 
Hence we can recover $\pr \phi$ from $U$ and its complex conjugate as follows, 
\be\label{equality1}
\begin{split}
\p_t \phi  &  = \frac{U(t)+ \overline{U(t)}}{2}, \quad \nabla \phi  = \frac{\nabla\big(U(t)- \overline{U(t)}\big) }{i 2 |\nabla| }. 
 \end{split}
\ee
Moreover, we also have 
\be\label{equality2}
\p_t V(t)   = e^{i t |\nabla| }(\p_t + i  |\nabla|) U(t) = e^{i t|\nabla|} \Box  \phi. \\ 
\ee
For the homogeneous case, $V(t) {=V(0)}$ plays the role of the initial data.  {  Thus   $V(t)$ is only useful in the inhomogenuous case.   }

 On the Fourier side, from \eqref{aug14eqn32}, we have
\[
 U(t,x)=\int_{\R^3 } e^{i x\cdot \xi - i t|\xi|} \widehat{V}(t,\xi) d \xi,\quad    \overline{U}(t,x)=\int_{\R^3 } e^{i x\cdot \xi + i t|\xi|} \widehat{\overline{V}}(t,\xi) d \xi. 
\]

Inspired by stationary phase analysis, we aim to exploit the oscillation of the phase in the radial direction. Since the function $\xi \mapsto |\xi|$ is homogeneous of degree one, we have 
\be\label{aug25eqn2}
 \forall\mu\in\{+,-\},\quad e^{i x\cdot \xi + i t\mu  |\xi| }= \frac{\xi\cdot \nabla_\xi}{i x\cdot \xi + i t\mu  |\xi|}\big( e^{i x\cdot \xi + i  t\mu |\xi|} \big).
\ee
Owing to this  homogeneity property, the radial derivative does not introduce any further singularities in the integration by parts procedure. Consequently, we have
\[
\xi\cdot \nabla_\xi( \frac{1}{i x\cdot \xi + i t\mu  |\xi|})= \frac{-1}{i x\cdot \xi + i t\mu |\xi|}. 
\]
To assess the severity of the potential singularity in the denominator of \eqref{aug25eqn2}, we must analyze the localized set $S_{k,l}(t,x)$ defined in \eqref{mainestvol}. This will be the focus of the following subsection.

\subsection{Measure estimate of the dyadically localized set}

\begin{lemma}\label{mainmeasureestlemma}
For the localized set $S_{k,l}(t,x)$ defined in \eqref{mainestvol}, the following measure estimate holds, 
\be\label{measureestlemma}
|S_{k,l}(t,x)|\lesssim 2^{3k+l}.  
\ee
\end{lemma}
\begin{proof}
Denote $\beta:=|x|/t.$ Note that 
\be\label{aug25eqn1}
1+ \frac{\mu x\cdot\xi}{t|\xi|}  = 1+\mu \beta \cos(\angle(x, \xi  )). 
\ee
The desired estimate \eqref{measureestlemma} is trivial if $l\geq-100$ because
\[
|S_{k,l}|\leq |\{\xi: \xi \in \R^3,  |\xi|\sim 2^k\}|\lesssim 2^{3k}. 
\]
Now, it would be sufficient to consider the case $l\leq -100$.   For this case we have   $|\beta|\geq 1/2$ because otherwise we have $| 1+\mu \beta \cos(\angle(x, \xi))|\geq 2^{-1} $, which implies that $S_{k,l}=\emptyset$ if $l\leq -100.$

Let $\{x/|x|, \mathbf{e}_1, \mathbf{e}_2\}$ be an fixed orthonormal frame. In terms of spherical coordinate,  we have $\xi=|\xi|\big(\sin \theta \cos\phi \mathbf{e}_1 + \sin \theta \sin\phi \mathbf{e}_2 -\mu \cos\theta x/|x|\big)$. Let
\[
\Theta_{l}:=\big\{\theta: |1-  \beta \cos(\theta) |\leq  2^{l}, \theta\in [0,\pi]   \big\}.
\]
Due to the fact that $\beta\geq 0$ and $l\leq -100.$ It's obvious that $\Theta_{l}\subset[0,\pi/2]$.
Therefore, in terms of the spherical coordinate, we have
\[
\int_{S_{k,l}}1  d \xi \lesssim \int_{c 2^{k-2}}^{C 2^{k+2}} \int_{\Theta_{l}} r^2 \sin \theta d r  d \theta\lesssim 2^{3k}  \int_{\Theta_{l}}  \sin \theta  d \theta.
\]
To sum up, to prove our desired estimate \eqref{measureestlemma},  it would be sufficient to prove the following estimate, 
\be\label{2025July28eqn81}
\textup{(Claim)}:\qquad   \int_{ \Theta_{l}} \sin \theta d \theta  \lesssim 2^l.
\ee
  Note that
\[
\int_{ \Theta_{l}\cap [0,2^{l/2+10} ]  } \sin \theta d \theta \lesssim  \int_{  [0,2^{l/2+10} ]  } \sin \theta d \theta = \int_0^{2^{l/2+10}} \sin \theta d \theta   \lesssim 2^l. 
\]
It remains to prove the following estimate 
\[
\int_{ \Theta_{l}\cap [2^{l/2+10}, \pi/2]  } \sin \theta d \theta  \lesssim 2^l
\]
Note that, if $0<x\ll 1$, we have 
\[
\sin x \geq \frac{x}{2}, \quad \Longrightarrow \big|\cos(2x) - \cos(x)\big|  =\int_{x}^{2x}  \sin \theta d \theta \geq \frac{3x^2}{4}. 
\]
If there exists $\theta_1, \theta_2 \in \Theta_{l}\cap [2^{l/2+10}, 2^{-10}]  $, s.t, $\theta_1 \geq 2\theta_2$, then from the monotonicity of $\cos x$ and   the above estimate, we have
\[
\big|\cos(\theta_1) - \cos(\theta_2)\big| = \cos\theta_2 -\cos\theta_1 \geq \cos\theta_2 -\cos(2\theta_2)\geq \frac{3(\theta_2)^2}{4}\geq \frac{3}{4} 2^{l+20}.
\]
On the other hand, since $\theta_1, \theta_2\in \Theta_{l},$ and $|\beta|>1/2$,  we have
\[
\big|\cos(\theta_1) - \cos(\theta_2)\big|= \beta^{-1}\big|(1-\beta\cos\theta_1)-(1-\beta\cos\theta_2) \big|\leq \beta^{-1} 2^{l+1}\leq 2^{l+2}.
\]
Hence we arrive at a  contradiction.

 To sum up,  we conclude
 \[
\forall \theta_1, \theta_2 \in \Theta_{l}\cap [2^{l/2+10}, 2^{-10}],\quad  \theta_1/2\leq \theta_2\leq 2\theta_1. 
\]
 With the above proved fact and the Taylor's theorem,  for any $\theta_1, \theta_2 \in \Theta_{l}\cap [2^{l/2+10}, 2^{-10}]  $, we have
\[
|\sin \theta_1||\theta_1-\theta_2|\leq 2^{4} \beta|\cos\theta_1-\cos\theta_2|= 2^4 \big|(1-\beta\cos\theta_1)-(1-\beta\cos\theta_2) \big| \leq 2^{l+5}. 
\]
Therefore, after fixing an arbitrary  $\theta^\star \in  \Theta_{l}\cap [2^{l/2+10}, 2^{-10}] $, we have 
\[
\int_{ \Theta_{l}\cap [2^{l/2+10}, 2^{-10}]  } \sin \theta d \theta  \lesssim  \int_{\theta^\star -2^{l+5}/\sin \theta^{\star}}^{\theta^\star + 2^{l+5}/\sin \theta^{\star}} \sin \theta^{\star} d \theta \lesssim 2^l. 
\]

Lastly, we estimate the contribution from the set $\Theta_{l}\cap [ 2^{-10}, \pi/2]$. Note that $\sin x$ is bounded from the above and the below in the set $[ 2^{-10}, \pi/2]$. Therefore,  for   any $\theta_1, \theta_2 \in \Theta_{l}\cap [ 2^{-10}, \pi/2]$, by   Lagrange mean value theorem and the fact that $\beta >1/2$, we have
\[
 |\theta_1-\theta_2|\lesssim  \beta|\cos\theta_1-\cos\theta_2|= \big|(1-\beta\cos\theta_1)-(1-\beta\cos\theta_2) \big|\leq 2^{l+5}
\]
 Therefore, again after fixing an arbitrary  $\theta_\star \in  \Theta_{l}\cap [2^{l/2+10}, 2^{-10}] $, we have 
\[
\int_{ \Theta_{l}\cap [ 2^{-10}, \pi/2]  } \sin \theta d \theta  \leq  \int_{\theta_\star -C 2^{l} }^{\theta^\star + C 2^{l} } d \theta \lesssim 2^l. 
\]
To sum up, we have
\[
  \int_{ \Theta_{l}} \sin \theta d \theta  \lesssim 2^l.
\]
Hence finishing the proof of the Claim in \eqref{2025July28eqn81}, which implies further the estimate    \eqref{measureestlemma}. 
\end{proof}
\subsection{Proof of \eqref{2025July23eqn10}}

For $k,k_1,k_2\in\Z,$  we use $ {P}_k $ and $P_{[k_1,k_2]}$ to denote  the Fourier multiplier operator with symbol $\chi_k(\xi)$ and $\chi_{[k_1,k_2]}(x)$ respectively.  Recall \eqref{equality1}. From the partition of unity, we have 
\be\label{2025July29eqn1}
|\pr \phi(t,x)| \lesssim \sum_{k\in \Z} |P_k U(t,x)|. 
\ee 
 {In view of the definition of $V$, see  \eqref{aug14eqn32}, we have  $S V(t)=  e^{i t |\nabla|} SU(t)$\footnote{ The   step relies on the fact that the operator $S$ commutes with $e^{it \nabla}$, i.e., $[S, e^{it \nabla}] = 0$.  This follows immediately from the identity $(t\partial_t - \xi\cdot \nabla_\xi)e^{i t|\xi|} = 0$, obtained by direct computation }. Therefore, setting $V^{S}(t,x):=  e^{i t |\nabla|} SU(t)$, we infer that   
\be\label{aug14eqn31}
\begin{split}
  \xi\cdot\nabla_\xi \widehat{V}(t, \xi)&= -\widehat{SV}(t, \xi) + t\p_t \hat{V}(t,\xi)-3\hat{V}(t,\xi)  \\
  &=-    \widehat{V^S}(t, \xi) + t\p_t \widehat{V}(t, \xi)- 3\widehat{V}(t, \xi)\\ 
  &= -    \widehat{V^S}(t, \xi) + t e^{it |\xi|} \widehat{\square\phi}(t,\xi)- 3\widehat{V}(t, \xi). 
\end{split}
\ee

 From the definition of $V(t)$ in \eqref{aug14eqn32}, we have 
\[
P_k U(t, x) = \int_{\R^3} e^{i x\cdot \xi - i t|\xi|}\widehat{V}(t, \xi)\chi_k(\xi ) d \xi.
\]

Let $t\in [2^{m-1}, 2^m], m \in \Z_{+}$, $\bar{l}:=-m-k$.   {Motivated by  the stationary phase analysis, we use the following inhomogeneous cutoff functions with a chosen threshold $\bar{l}$ to quantify the smallness of the denominator in \eqref{aug25eqn2},}
\be\label{cutoff}
\varphi^l_{\bar{l}}(\xi, x,t):=\left\{\begin{array}{cc}
 \displaystyle{\chi_{\leq -m-k}\big(1- \frac{  x\cdot\xi}{t|\xi|} \big) } & \textup{if\quad} l=\bar{l}, \\
  \displaystyle{\chi_{l}\big(1- \frac{  x\cdot\xi}{t|\xi|} \big) } &  \textup{\,\,\,\, if\quad } l\in (\bar{l},2),\\
   \displaystyle{\chi_{\geq 2}\big(1- \frac{  x\cdot\xi}{t|\xi|} \big) } & \textup{if\quad} l=2, \\
\end{array}\right.\qquad
\ee
Note that  we have 
\beaa
  \sum_{l\in [\bar{l}, 2]\cap \Z}\varphi^l_{\bar{l}}(\xi, x,t) =1.
 \eeaa
 From the above partition of unity,  we decompose $U(t, x)$ as follows, 
\be\label{decomgeneralI}
\begin{split}
P_k(U) (t, x) &=\sum_{l\in [\bar{l}, 2]\cap \Z} J_{l;\bar{l}}(t, x),\qquad \bar{l}:=-m-k, \\\
 J_{l;\bar{l}}(t, x)&=  \int_{\R^3} e^{i x\cdot \xi - i t |\xi|} \widehat{V}(t,\xi)\chi_k(\xi ) \varphi^l_{\bar{l}}(\xi, x,t) d\xi.   \\
\end{split}
\ee
Based on the size of $l$, we proceed in two steps as follows.

\vo 

\textbf{Step 1.}\quad If $l=\bar{l}$. 

\vo 
 
Recall the definition of    $\bar{l}:=-m-k$, and the fact that  $t\in [2^{m-1}, 2^m], m \in \Z_{+}$. Note that, after first using the Cauchy-Schwarz inequality and then using   the measure estimate \eqref{measureestlemma} in  Lemma \ref{mainmeasureestlemma},  {the definition of $V$ in \eqref{aug14eqn32} } and the equality \eqref{equality1},   we have
\be\label{thresholdest}
\begin{split}
 \big|J_{\bar{l};\bar{l}}(t, x)\big|&\lesssim (2^{3k+\bar{l}})^{1/2}\| \widehat{V}(t,\xi) \chi_k(\xi ) \|_{L^2_\xi } = 2^{-m/2+k} \| e^{it |\xi|} \widehat{U}(t,\xi) \chi_k(\xi ) \|_{L^2_\xi }\\ 
 &\lesssim |t|^{-1/2 } 2^{ k} \big( \|\widehat{\p_t\phi}(t, \xi)  \chi_k(\xi )  \|_{L^2_\xi} +   \|\widehat{\nabla \phi}(t, \xi)  \chi_k(\xi )  \|_{L^2_\xi} \big) \\
 &\lesssim  |t|^{-1/2 } 2^{k} \| P_k(\pr \phi)(t,x)\|_{L^2_x} \\ 
  &\lesssim  |t|^{-1/2 } \big[2^{k} \| P_k(\pr \phi)(t,x)\|_{L^2_x}  \mathbf{1}_{k\leq 0} + 2^{k-2k}  \| P_k(\pr \nabla^2 \phi)(t,x)\|_{L^2_x} \mathbf{1}_{k\geq 0} \big] \\ 
  &\lesssim  |t|^{-1/2 } 2^{k-2\max\{k,0\}} \| P_k(\pr \nab^{\leq 2} \phi)(t,x)\|_{L^2_x}   \\ 
 &\lesssim |t|^{-1/2 } 2^{-|k|} \| P_k(\pr\Gamma^{\leq 2}\phi)(t,x)\|_{L^2_x}.
 \end{split}
\ee
 
Moreover, if $|x|\leq t-1$, we can improve the above estimate by using the equality \eqref{aug25eqn2} and doing integration by parts in $\xi$ once along the radial direction. As a result, we have  
\be\label{2025July29eqn11}
J_{\bar{l};\bar{l}}(t, x)= - \int_{\R^3} e^{i x\cdot \xi - i  t |\xi|} (3+ \xi\cdot \nabla_\xi)\big(   \frac{\widehat{V}(t,\xi)\chi_k(\xi )}{i x\cdot \xi - i t  |\xi|}   \big) \varphi^l_{\bar{l}}(\xi, x,t) d\xi. 
\ee
In the above estimate, we used the fact that 
\[
 \xi\cdot\nabla_\xi\big(\varphi^l_{\bar{l}}(\xi, x,t)\big)=0, 
\]
which follows from the direct computation, see \eqref{cutoff}. Note that, due to the assumption $|x|\leq t-1$, we have
\be\label{2025July29eqn31}
  | x\cdot \xi -   t  |\xi||\gtrsim |\xi|(|t|-|x|). 
\ee
From the above estimate, the last equality in \eqref{aug14eqn31}, the Cauchy-Schwarz inequality and then use  the measure estimate \eqref{measureestlemma} in  Lemma \ref{mainmeasureestlemma},   we have
\be\label{2025July29eqn12}
\begin{split}
 \big|J_{l;\bar{l}}(t, x)\big|& \lesssim   (|t|-|x|)^{-1} 2^{-k} (2^{3k+\bar{l}})^{1/2} \big[\|\xi\cdot \nabla_\xi \widehat{V} (t,\xi) \|_{L^2_\xi} +  \|\widehat{V}(t,\xi) \|_{L^2_\xi}\big]\\ 
 &\lesssim (|t|-|x|)^{-1} |t|^{-1/2}\big[ \|  \pr \Gamma^{\leq 1}\phi (t,x)\|_{L^2_x} +t\|\square \phi (t,x)\|_{L^2_x}\big]. 
\end{split}
\ee
After combining  \eqref{thresholdest} and the above estimate \eqref{2025July29eqn12}, we have
\be\label{2025July29eqn13}
   \big|J_{l;\bar{l}}(t, x)\big|  \lesssim  2^{-\delta |k|} |t|^{-1/2}  \big[ \langle |t|-|x|\rangle ^{-1+\delta} \mathbf{1}_{|x|\leq t} +\mathbf{1}_{|x| > t}  \big] \big[ \| \pr \Gamma^{\leq 2}\phi (t,x)\|_{L^2_x}+t\| \square \phi (t,x)\|_{L^2_x}\big].
\ee

\textbf{Step 2.}\quad If $l\in (\bar{l},2]\cap\Z$.

Recall the definition of cutoff functions in \eqref{cutoff}. Note that  
\be\label{aug25eqn3}
\forall \xi\in supp(\varphi^l_{\bar{l}}(\xi, x,t)\chi_k(\xi)), \qquad |x\cdot \xi -   t |\xi||\gtrsim 2^{m+k+l}.
\ee
We first do integration by parts in $\xi$ once. As a result, we yield the same formula as in \eqref{2025July29eqn11}.  From the equality of $  \xi\cdot\nabla_\xi \widehat{V}(t, \xi)$ in \eqref{aug14eqn31}, we first decompose $J_{\bar{l};\bar{l}}(t, x)$ in \eqref{2025July29eqn11} into two pieces as follows, 
\be\label{2025July29eqn22}
\begin{split}
J_{\bar{l};\bar{l}}(t, x) & = J^1_{\bar{l};\bar{l}}(t, x)+ J^2_{\bar{l};\bar{l}}(t, x), \\ 
J^1_{\bar{l};\bar{l}}(t, x)& =  - \int_{\R^3} e^{i x\cdot \xi - i  t  |\xi|}\big[ ( \xi\cdot \nabla_\xi)\big( \frac{  \chi_k(\xi )}{i x\cdot \xi - i t  |\xi|}   \big) \widehat{V}(t,\xi) +   \frac{ \widehat{V^S}(t, \xi) \chi_k(\xi )}{i x\cdot \xi - i t  |\xi|}   \big] \varphi^l_{\bar{l}}(\xi, x,t) d\xi, \\  
J^2_{\bar{l};\bar{l}}(t, x)& =   - t \int_{\R^3} e^{i x\cdot \xi }      \widehat{\square \phi }(t, \xi)  \big( \frac{  \chi_k(\xi )}{i x\cdot \xi - i t |\xi|}   \big)  \varphi^l_{\bar{l}}(\xi, x,t) d\xi. \\ 
\end{split}
\ee
For $J^i_{\bar{l};\bar{l}}(t, x),i\in\{1,2\}$, after first using the Cauchy-Schwarz inequality and then using   the measure estimate \eqref{measureestlemma} in  Lemma \ref{mainmeasureestlemma},  from \eqref{aug25eqn3}, we have
\be\label{2025July29eqn23}
\begin{split}
\sum_{l\in (\bar{l},2]\cap \Z}\big| J^1_{l;\bar{l}}(t, x)\big|
&\lesssim \sum_{l\in (\bar{l},2]\cap \Z} 2^{-m-(k+l)} 2^{(3k+l)/2}\big(\|  \widehat{V^S} (t,\xi)   \|_{L^2}+ \|    \widehat{V}(t,\xi)    \|_{L^2}\big) \\ 
&\lesssim |t|^{-1/2} 2^{-|k|}  \|  \pr \Gamma^{\leq 3}\phi (t,x)\|_{L^2_x}  \\
\sum_{l\in (\bar{l},2]\cap \Z}\big| J^2_{l;\bar{l}}(t, x)\big|
&\lesssim \sum_{l\in (\bar{l},2]\cap \Z} 2^{-m-(k+l)} 2^{(3k+l)/2} |t|\|  \widehat{\square \phi} (t,\xi)   \|_{L^2} \\ 
&\lesssim |t|^{1/2} 2^{-|k|}  \|\square \Gamma^{\leq 2}\phi (t,x)\|_{L^2_x}. 
\end{split}
\ee

Now, for the case $|x|\leq t-1$, we wish to prove an improved decay  rate of the distance with respect to the light cone for $J_{\bar{l};\bar{l}}^1(t, x)$. We do integration by parts in ``$\xi$'' in radial direction again for $J_{\bar{l};\bar{l}}^1(t, x)$ in the formula \eqref{2025July29eqn22}. As a result, we have 
\be\label{2025July29eqn63}
\begin{split}
J^1_{\bar{l};\bar{l}}(t, x)& =  - \int_{\R^3} e^{i x\cdot \xi - i  t  |\xi|} (3+ \xi\cdot \nabla_\xi)\Big[ \frac{1}{i x\cdot \xi - i t |\xi|} \big[  \frac{ \widehat{V^S}(t, \xi) \chi_k(\xi )}{i x\cdot \xi - i t  |\xi|}  \\ 
&\quad +   ( \xi\cdot \nabla_\xi)\big( \frac{  \chi_k(\xi )}{i x\cdot \xi - i t  |\xi|}   \big) \widehat{V}(t,\xi) \big] \Big] \varphi^l_{\bar{l}}(\xi, x,t) d\xi. \\ 
\end{split}
\ee

After first using the Cauchy-Schwarz inequality  and the equality \eqref{aug14eqn31},  then using   the measure estimate \eqref{measureestlemma} in  Lemma \ref{mainmeasureestlemma},  from \eqref{aug25eqn3} and \eqref{2025July29eqn31}, we have  
\be\label{2025July29eqn43}
\begin{split}
\sum_{l\in (\bar{l},2]\cap \Z}\big| J^1_{l;\bar{l}}(t, x)\big|
&\lesssim \sum_{l\in (\bar{l},2]\cap \Z} (t-|x|)^{-1}2^{-k} 2^{-m-(k+l)} 2^{(3k+l)/2}  \big(\|  \widehat{V^S} (t,\xi)   \|_{L^2}\\ 
&\quad+ \|    \widehat{V}(t,\xi)     \|_{L^2} + \| \xi\cdot\nabla_\xi \widehat{V^S} (t,\xi)   \|_{L^2}+ \|  \xi\cdot\nabla_\xi   \widehat{V}(t,\xi)     \|_{L^2} \big) \\ 
&\lesssim |t|^{-1/2}  (t-|x|)^{-1}\big[ \|  \pr \Gamma^{\leq 2}\phi (t,x)\|_{L^2_x} + |t|  \|\square \Gamma^{\leq 1}\phi (t,x)\|_{L^2_x} \big]. \\
\end{split}
\ee
After combining the obtained estimates \eqref{2025July29eqn23} and \eqref{2025July29eqn43}, we conclude
\be\label{2025July29eqn51}
\begin{split}
\sum_{l\in (\bar{l},2]\cap \Z}\big| J_{l;\bar{l}}(t, x)\big|&
\lesssim2^{-\delta |k|} \big[ |t|^{-1/2} \big(\langle |t|-|x|\rangle ^{-1+\delta} \mathbf{1}_{|x|\leq t} +\mathbf{1}_{|x| > t}  \big)    \| \pr \Gamma^{\leq 3}\phi (t,x)\|_{L^2_x}\\ 
&\quad + |t|^{1/2}  \| \square \Gamma^{\leq 2} \phi (t,x)\|_{L^2_x}\big].
\end{split}
\ee
Recall \eqref{2025July29eqn1}  and  \eqref{decomgeneralI}. To sum up, from the obtained estimate  \eqref{2025July29eqn13} and \eqref{2025July29eqn51},   {in view of \eqref{2025July29eqn1} and  \eqref{decomgeneralI}},  we conclude 
\[
  |\pr \phi(t,x)| \lesssim |t|^{-1/2} \big(\langle |t|-|x|\rangle ^{-1+\delta} \mathbf{1}_{|x|\leq t} +\mathbf{1}_{|x| > t}  \big)    \| \pr \Gamma^{\leq 3}\phi (t,x)\|_{L^2_x}+ |t|^{1/2}  \| \square \Gamma^{\leq 2} \phi (t,x)\|_{L^2_x}.
\]
Hence finishing the proof of the desired estimate \eqref{2025July23eqn10}. 
 

\subsection{Proof of \eqref{2025July23eqn11}}

  
For consistency, we maintain the notation established in the previous subsection. The key distinction between the proofs of \eqref{2025July23eqn10} and \eqref{2025July23eqn11} is that we now treat certain terms in $L^\infty_\xi$ space rather than $L^2_\xi$.
 
 Again we use the decomposition \eqref{decomgeneralI} for $P_k(U) (t, x) $. Based on the size of $l$, we proceed in two steps as follows.

\vo 

\textbf{Step 1.}\quad If $l=\bar{l}$. 

\vo 

After using   the measure estimate \eqref{measureestlemma} in  Lemma \ref{mainmeasureestlemma},   {for $t\in[ 2^{m-1}, 2^m]$},   we have
\be\label{2025July29eqn56}
\big|J_{\bar{l};\bar{l}}(t, x)\big|\lesssim 2^{3k+\bar{l}}  \| \widehat{V}(t,\xi)   \|_{L^\infty_\xi }\les { 2^{2k-m} \| \widehat{\pr \phi}(t,\xi)\|_{L^\infty_\xi}}
\lesssim |t|^{-1}2^{2k} \| \widehat{\pr \phi}(t,\xi)\|_{L^\infty_\xi}.
\ee 

\vo 

\textbf{Step 2.}\quad  If $l\in (\bar{l},2]\cap\Z$. 

\vo 

For this case, we first do integration by parts in $\xi$ once and then use the decomposition in \eqref{2025July29eqn22}. For $J^2_{\bar{l};\bar{l}}(t, x)$, see \eqref{2025July29eqn22}, after using the Cauchy-Schwartz inequality and   the measure estimate \eqref{measureestlemma} in  Lemma \ref{mainmeasureestlemma},   we have 
\be\label{2025July29eqn66}
\sum_{l\in (\bar{l},2]\cap \Z}\big| J^2_{l;\bar{l}}(t, x)\big|
 \lesssim \sum_{l\in (\bar{l},2]\cap \Z} 2^{-m-(k+l)} 2^{(3k+l)/2} |t|\|  \widehat{\square \phi} (t,\xi)   \|_{L^2}\lesssim |t|^{1/2} 2^k \|   {\square \phi}     \|_{L^2}. \\  
\ee

For $J^1_{\bar{l};\bar{l}}(t, x)$,  see \eqref{2025July29eqn22}, we do integration by parts in $\xi$ again. As a result, we yield the same formula as in \eqref{2025July29eqn63}. Moreover, we decompose $J^1_{\bar{l};\bar{l}}(t, x)$ in \eqref{2025July29eqn63} into three pieces as follows,
\be\label{2025July30eqn21}
J^1_{\bar{l};\bar{l}}(t, x)  = J^{1;1}_{\bar{l};\bar{l}}(t, x)+ J^{1;2}_{\bar{l};\bar{l}}(t, x)+ J^{1;3}_{\bar{l};\bar{l}}(t, x),
\ee  
where
\be\label{2025July29eqn691}
\begin{split}
J^{1;1}_{\bar{l};\bar{l}}(t, x)& =  - \int_{\R^3} e^{i x\cdot \xi - i  t  |\xi|}\Big[\widehat{V^S}(t, \xi) (3+ \xi\cdot \nabla_\xi)\big( \frac{\chi_k(\xi )}{(i x\cdot \xi - i t |\xi|)^2} \big)  \\ 
&\quad +    (3+ \xi\cdot \nabla_\xi)\big(  \frac{  1}{i x\cdot \xi - i t  |\xi|}  ( \xi\cdot \nabla_\xi)\big( \frac{  \chi_k(\xi )}{i x\cdot \xi - i t  |\xi|}   \big)\big) \widehat{V}(t,\xi) \big] \Big] \varphi^l_{\bar{l}}(\xi, x,t) d\xi, \\
\end{split}
\ee
\be\label{2025July31eqn1}
\begin{split}
J^{1;2}_{\bar{l};\bar{l}}(t, x)& =  - \int_{\R^3} e^{i x\cdot \xi - i  t  |\xi|}\Big[(\xi\cdot \nabla_\xi-t\p_t)\widehat{V^S}(t, \xi)   \big( \frac{\chi_k(\xi )}{(i x\cdot \xi - i t |\xi|)^2} \big)  \\ 
&\quad +    (\xi\cdot \nabla_\xi-t\p_t)\widehat{V}(t,\xi) \big(  \frac{  1}{i x\cdot \xi - i t  |\xi|}  ( \xi\cdot \nabla_\xi)\big( \frac{  \chi_k(\xi )}{i x\cdot \xi - i t  |\xi|}   \big)\big)  \big] \Big] \varphi^l_{\bar{l}}(\xi, x,t) d\xi, \\  
\end{split}
\ee
\be\label{2025July31eqn2}
\begin{split}
J^{1;3}_{\bar{l};\bar{l}}(t, x)& =  - \int_{\R^3} e^{i x\cdot \xi - i  t  |\xi|}\Big[ t\p_t\widehat{V^S}(t, \xi)   \big( \frac{\chi_k(\xi )}{(i x\cdot \xi - i t |\xi|)^2} \big)  \\ 
&\quad +   t\p_t\widehat{V}(t,\xi) \big(  \frac{  1}{i x\cdot \xi - i t  |\xi|}  ( \xi\cdot \nabla_\xi)\big( \frac{  \chi_k(\xi )}{i x\cdot \xi - i t  |\xi|}   \big)\big)  \big] \Big] \varphi^l_{\bar{l}}(\xi, x,t) d\xi. \\  
\end{split}
\ee
After using the Cauchy-Schwartz inequality and   the measure estimate \eqref{measureestlemma} in  Lemma \ref{mainmeasureestlemma},   we have 
\be\label{2025July29eqn73}
\begin{split}
\sum_{l\in (\bar{l},2]\cap \Z}\big| J^{1;3}_{l;\bar{l}}(t, x)\big|&
 \lesssim \sum_{l\in (\bar{l},2]\cap \Z} 2^{-2m-2(k+l)} 2^{(3k+l)/2} |t|\|  \widehat{\square\Gamma^{\leq 1} \phi} (t,\xi)   \|_{L^2}\\ 
 &\lesssim |t|^{1/2} 2^k \|   {\square \Gamma^{\leq 1} \phi}     \|_{L^2}. 
 \end{split}  
\ee
For $J^{1;1}_{\bar{l};\bar{l}}(t, x)$ and $J^{1;2}_{\bar{l};\bar{l}}(t, x)$, we use the full volume of the localized set by putting in $\widehat{V}(t,\xi)$ etc  in  $L^\infty_\xi$. From  the measure estimate \eqref{measureestlemma} in  Lemma \ref{mainmeasureestlemma},   we have
\be\label{2025July29eqn78}
\begin{split}
\sum_{l\in (\bar{l},2]\cap \Z}\big| J^{1;1}_{l;\bar{l}}(t, x)\big| + \big| J^{1;2}_{l;\bar{l}}(t, x)\big| &
 \lesssim \sum_{l\in (\bar{l},2]\cap \Z} 2^{-2m-2(k+l)} 2^{(3k+l)}  \|  \widehat{\pr \Gamma^{\leq 2} \phi} (t,\xi)   \|_{L^\infty_\xi}\\ 
 &\lesssim  |t|^{-1}2^{2k} \| \widehat{\pr  \Gamma^{\leq 2} \phi}(t,\xi)\|_{L^\infty_\xi}.
 \end{split}  
\ee 

Recall the decompositions in \eqref{decomgeneralI}, \eqref{2025July29eqn22}, and \eqref{2025July30eqn21}. To sum up, after combining the obtained estimates \eqref{2025July29eqn56}, \eqref{2025July29eqn66}, \eqref{2025July29eqn73}, and \eqref{2025July29eqn78}, we have
\be\label{2025Aug16eqn1}
  |P_k(U) (t, x)| \lesssim|t|^{-1}2^{2k} \| \widehat{\pr  \Gamma^{\leq 2} \phi}(t,\xi)\|_{L^\infty_\xi}+  |t|^{1/2} 2^k \|   {\square \Gamma^{\leq 1} \phi}     \|_{L^2}.
\ee
Recall \eqref{aug14eqn32}.  In view of \eqref{2025July29eqn1},  our  desired estimate \eqref{2025July23eqn11} follows from  the estimate \eqref{2025Aug16eqn1} obtained above.

  \section{Proof of Theorem \ref{Thm:main3}}  
  To prove our main theorem, we use the standard bootstrap argument. Firstly, $\forall i\in \{1,\cdots, m\},$ we define  
\[
\begin{split}
 E_i(t)&:= \|   \pr \Gamma^{\leq 30} \phi_i\|_{L^2}, \quad 
N^{high}_i(t) := \|   \pr \Gamma^{\leq 30} N_i\|_{L^2}, \quad N^{low}_i(t) :=  \|   \pr \Gamma^{\leq 25} N_i\|_{L^2}.  \\ 
\end{split}
\]
 We  make the following bootstrap assumption
\be\label{bootstrapfirst}
\sup_{t\in [0,T]} \sum_{i\in \{1,\cdots, m\}} E_i(t) +  \langle t \rangle^{3/2-2\delta } N^{high}_i(t)+   \langle t \rangle^{2-9\delta } N^{low}_i(t)\leq \epsilon_1:=\epsilon_0^{3/4}. 
\ee
From the decay estimate \eqref{finalestphysicspace} in Theorem  \ref{Thm:main1}, we have
\be\label{2025June14eqn61} 
\begin{split}
  |\pr\Gamma^{\leq 20} \phi_i(t,x)|&  \lesssim \langle t\rangle^{-(1-\delta)/2} \langle t-r_i\rangle^{-1/2+\delta} \epsilon_1+ \langle t \rangle^{-3/2+10\delta}\epsilon_1,\\ 
  |\pr\Gamma^{\leq 25} \phi_i(t,x)| & \lesssim \langle t\rangle^{-(1-\delta)/2 } \langle t-r_i\rangle^{-1/2+\delta} \epsilon_1+ \langle t \rangle^{-1+3\delta}\epsilon_1.\\ 
  \end{split}
\ee

With the above preparation, we are ready to close the bootstrap argument. We proceed in three steps as follows. 

\vo 

\textbf{Step 1.} The estimate of $N^{high}_i(t)$. 

\vo

Note that, from the separation assumption, i.e., \eqref{2025June14eqn41}, the following estimate always holds, 
\[
 \forall x\in \R^3,  i,j\in \{1,\cdots, m\}, i\neq j, \quad   {\langle t-r_i\rangle}^{-1}    {\langle t-r_j\rangle}^{-1} \lesssim \langle t\rangle^{-1}. 
\]
Let $0< \delta \ll 1$.  From  the no self-interaction assumption,   the above estimate,  the   first decay estimate in \eqref{2025June14eqn61},   { and estimating the nonlinear terms  in such a way that  the wave with most vector fields is taken in  $L^2$ and the other two waves in $L^\infty$},  we have
\be\label{2025June17eqn21}
\sum_{i\in \{1,\cdots, m\}}  \|   \pr \Gamma^{\leq 30} N_i\|_{L^2} \lesssim \epsilon_1^3  \langle t \rangle^{-3/2+2\delta}\lesssim  \langle t \rangle^{-3/2+2\delta}\epsilon_0^2.
\ee

\vo 

\textbf{Step 2.}   The estimate of $E_i(t)$.

\vo

From the obtained estimate \eqref{2025June17eqn21} and the standard energy estimate, we conclude
\be\label{2025July29eqn101}
  \sum_{i\in \{1,\cdots, m\}}  E_i(t)^2 \leq  \sum_{i\in \{1,\cdots, m\}}  E_i(0)^2 + \int_0^t \langle s \rangle^{-3/2+2\delta}\epsilon_1\epsilon_0^2 ds \lesssim \epsilon_0^2.
\ee

\vo 

\textbf{Step 3.} The estimate of $N^{low}_i(t)$.

\vo

The estimates for $N^{low}_i(t)$ and $N^{high}_i(t)$ differ significantly in their treatment of the wave with the most vector fields. For $N^{low}_i(t)$, this wave can be placed in $L^\infty$ due to a strategically chosen gap between the total number of derivatives used in the estimate and the maximum number of derivatives propagated. This is not possible for $N^{high}_i(t)$.

We divide the whole space $\R^n$ into $N+1$-regions as follows,
\[
  \R^3:= \cup_{i=0}^{N} R_i(t), \quad \forall i\in\{1,\cdots, N\},\quad R_i(t):=\{x: |x-r_i|\leq c t \}, \quad R_0:=(\cup_{i=1}^N R_i(t))^c,
\]
where the sufficiently small absolute constant $c$ only depends the fixed parameters $\{c_{i;j}, i\in\{1,\cdots, N\},j\in\{1,2,3\}\}$. In region $R_i(t)$, thanks to the separation condition \eqref{2025June14eqn41}, the decay rates of other waves,   {except  the one  corresponding to  $u_i$},  are improved, see the second estimate in \eqref{2025June14eqn61}. In region  $R_0(t)$, the decay rates of all waves   are improved. Therefore, thanks to the no-self interaction condition, in region $R_i(t)$, after putting any two waves other than  { the one corresponding to   $u_i$} in $L^\infty$ and the other one in $L^2$ regardless how many vector fields they have, from the second estimate in \eqref{2025June14eqn61}, we have
\be\label{2025June17eqn25}
  \sum_{i\in \{1,\cdots, m\}}  \|   \pr \Gamma^{\leq 25} N_i\|_{L^2} \lesssim \epsilon_1^3  \langle t \rangle^{-1+3\delta} \langle t \rangle^{-1+3\delta}\lesssim  \langle t \rangle^{-2+6\delta}\epsilon_0^2.
\ee
 {This improves  the  bootstrap argument and thus  end the proof of the theorem. } The scattering property follows from the corollary of the energy estimate, see \eqref{2025July29eqn101}. 

\section{Proof of Theorem \ref{Thm:main4}}
To prove our main theorem, we use the standard bootstrap argument, which     mainly relies on the decay estimate obtained in \eqref{2025July23eqn11} in Theorem \ref{Thm:main3}.  For $i\in \{1,\cdots, m\},$ we define   
\[
E_i(t):= \|   \pr \Gamma^{\leq 30}\phi_i\|_{L^2_x}, \quad W_i(t):= \|  \widehat{ \pr \Gamma^{\leq 30}\phi_i}\|_{L^\infty_\xi}.
\]
Moreover,  we define the $L^\infty_x$-type norm,
\[
K_i(t):=  \|   \pr \Gamma^{\leq 20}\phi_i\|_{L^\infty_x}.
\]

Let $0< \delta \ll 1$. We  make the following bootstrap assumption
\be\label{Bootstraptheorem2}
\sup_{t\in[0,T]} \sum_{i=1,\cdots, m} E_i(t) +  \langle t \rangle^{-\delta} W_i(t)+ \langle  t\rangle^{1-\delta} K_i(t)\leq \epsilon_1:=\epsilon_0^{3/4}. 
\ee

To close the bootstrap argument, we proceed in three steps as follows.

\vo 

\textbf{Step 1.}\quad {Energy estimate.}

\vo 

Since our nonlinearity is semilinear and cubic, from the $L^2_x-L^\infty_x-L^\infty_x$-type trilinear estimate and the energy estimate \eqref{energyeestmain}, we have 
 \be\label{aug15eqn72}
\forall i\in\{1,\cdots, m\},\quad E_i(t)-E_i(0)\lesssim \int_{0}^t \langle s\rangle^{-2+2\delta} \epsilon_1^3 d s\lesssim \epsilon_0.
\ee
 
\vo 

\textbf{Step 2.}\quad  {$L^\infty_\xi$ estimate}

\vo

From the Duhamel's formula and  the $L^2_x-L^2_x-L^\infty_x$-type trilinear estimate, we have
 \be\label{aug15eqn71}
 \begin{split}
  W_i(t)=  \|  \widehat{ \pr \Gamma^{\leq 30}\phi_i}(t,\xi) \|_{L^\infty_\xi}&\lesssim \epsilon_0 + \int_{0}^t \|  \, ^{(i)} \square\Gamma^{\leq 30}\phi_i \|_{L^1_x}d s \lesssim  \epsilon_0 + \int_{0}^t \langle s\rangle^{-1+\delta}\epsilon_1^3 d s\lesssim \langle t \rangle^{\delta}\epsilon_0.\\ 
 \end{split}
 \ee

\vo 

\textbf{Step 3.}\quad  {$L^\infty_x$-decay estimate}

\vo

Note that, from the $L^2_x-L^\infty_x-L^\infty_x$-type trilinear estimate, we have
\[
\forall t\in [0,T],\qquad   \|  \, ^{(i)} \square\Gamma^{\leq 30}\phi_i(t,x) \|_{L^2_x} \lesssim \langle t\rangle^{-2+2\delta}\epsilon_1^3.
\]
Therefore, from the above estimate, the improved energy estimate \eqref{aug15eqn72}, and the improved $W_a(t)$-estimate \eqref{aug15eqn71}, and  the decay estimate obtained in \eqref{2025July23eqn11} in Theorem \ref{Thm:main3}, we have
\[
K_a(t)\lesssim \langle t\rangle^{-1+\delta}\epsilon_0 + \langle t\rangle^{-3/2+2\delta}\epsilon_0 \lesssim \langle t\rangle^{-1+\delta}\epsilon_0. 
\]
Now, it's clear that the bootstrap assumption in \eqref{Bootstraptheorem2} is improved. The scattering property follows from the corollary of the energy estimate, see \eqref{aug15eqn72}.  Therefore, finishing the proof of Theorem \ref{Thm:main4}.

\end{document}